\documentclass[12pt]{article}
\usepackage[top=1in, bottom=1in, left=1in, right=1in]{geometry}

\usepackage{xcolor}
\usepackage{amssymb,bm}
\usepackage{amsmath}
\usepackage{amsthm}
\usepackage{graphicx}
\usepackage{hyperref}
\usepackage{caption}
\usepackage{float} 
\hypersetup{pdfborder=0 0 0}

\usepackage{physics}

\newtheorem{theorem}{{\sc Theorem}}[section]

\newtheorem{lemma}[theorem]{{\sc Lemma}}

\newtheorem{remark}[theorem]{Remark}
\newtheorem{definition}[theorem]{Definition}
\newtheorem{conjecture}[theorem]{Conjecture}
\newtheorem{notation}[theorem]{Notation}

\newcommand\Restr[2]{{ \left.\kern-\nulldelimiterspace #1 \vphantom{\big|} \right|_{#2}}}
\newcommand{\RR}{\mathbb{R}}

\newcommand{\CC}{\mathbb{C}}

\newcommand{\CB}{\mathcal{B}}

\newcommand{\CM}{\mathcal{M}}

\newcommand{\CJ}{\mathcal{J}}

\newcommand{\n}{\noindent}
\newcommand{\comment}[1]{}

\newcommand{\Ai}{\text{Ai}}

\catcode`,\active

\catcode`\,12

\title{On the distribution of Born transmission eigenvalues in the complex plane}
\author{Narek Hovsepyan\footnote{Department of Mathematics, Rutgers University, New Brunswick, NJ, USA (narek.hovsepyan@rutgers.edu)}}
\date{}

\date{}
\begin{document}
\maketitle

\begin{abstract}
We analyze an approximate interior transmission eigenvalue problem in ${\mathbb R}^d$ for $d=2$ or $d=3$, motivated by the transmission problem of a transformation optics-based cloaking scheme and obtained by replacing the refractive index with its first order approximation, which is an unbounded function. Using the radial symmetry we show the existence of (infinitely many) complex transmission eigenvalues and prove their discreteness. Moreover, it is shown that there exists a horizontal strip in the complex plane around the real axis, that does not contain any transmission eigenvalues. 
\end{abstract}

\section{Introduction}
\setcounter{equation}{0}

Transmission eigenvalue problems have become an important area of research in inverse scattering theory. We refer the reader to \cite{CCH21} for a recent survey article, or to \cite{CCH22, CK19} and the references therein. The relation of transmission eigenvalues to cloaking, or to so-called non-scattering wave numbers is noteworthy \cite{BPS14, EH18, CX21, VX21, CV}. At such wave numbers there exists an incident field which does not scatter, i.e. the obstacle becomes invisible, when probed by that incident field. In \cite{CHV23} we considered an approximate, transformation optics-based cloaking scheme for the Helmholtz equation that incorporated a Drude-Lorentz model \cite{J01, NV16} to account for the dispersive properties of the cloak. Specifically, the goal is to make $B_1$, the ball of radius 1 centered at the origin, approximately invisible to a far observer, independently of its contents. This is done using ideas from transformation optics \cite{GLU03, PSS06, KSVW08, GKLU09, KOV10, phys1, phys2} by surrounding the cloaked region $B_1$ with a layer of an appropriate anisotropic material that we assume occupies the annulus $B_2\backslash B_1$ and refer to it as the cloak, or the cloaking layer. The material properties of the cloak are obtained via change of variables using a map $F_\epsilon$ that blows up a small ball $B_\epsilon$ of radius $\epsilon>0$ to the cloaked region $B_1$, while keeping the outer boundary $\partial B_2$ fixed (this map is commonly used in many cloaking papers, e.g. \cite{KSVW08, KOV10, NV16}, but its precise formula is irrelevant here). Finally, by including a layer of extremely high conductivity adjacent to $B_1$, without loss of generality we assume that $B_1$ is ``soft", i.e. we impose a zero Dirichlet boundary condition on $\partial B_1$. The refractive index of this cloak is

\begin{equation} \label{n_eps}
n_\epsilon(x,k) = 1 + \sigma_\epsilon(k) \det DF_\epsilon(x),
\end{equation}

\n where

\begin{equation*}
\sigma_\epsilon(k) = \frac{1}{k_\epsilon^2 - k^2 - ik}
\qquad \text{and} \qquad
\det DF_\epsilon(x) = \frac{(2-2\epsilon + |x|)^{d-1}}{(2-\epsilon)^d |x|^{d-1}}.
\end{equation*}

\n Here $k$ denotes the wave number, $d=2$ or $d=3$ is the dimension and $k_\epsilon > 0$ represents the so-called resonant frequency of the Drude-Lorentz term. The interior transmission eigenvalue problem corresponding to this cloaking scheme is nonlinear and its complete analysis is an open problem: the existing approaches do not apply \cite{CCH22}. It turns out \cite{CHV23} that there are no real transmission eigenvalues, consequently perfect cloaking/non-scattering is impossible at any $k>0$. However, one can achieve approximate cloaking: under certain growth assumption on $k_\epsilon$, the scattered field of the cloak and its far field pattern can be made uniformly small (of order $\epsilon$ in 3d and $1/|\ln \epsilon|$ in 2d) on any finite band of wave numbers and any incident field, provided $\epsilon$ is sufficiently small. Consistent with this growth assumption we will suppose that 

\begin{equation*}
k_\epsilon^2 \sim \epsilon^{-d}, \qquad \qquad \epsilon \to 0.
\end{equation*}

\n Interestingly, transmission eigenvalues exhibit non-discreteness: supported with numerical evidence, we conjectured that sequences of complex transmission eigenvalues accumulate at two finite points in the complex plane, namely the poles of the Drude-Lorentz term $\sigma_\epsilon(k)$. 

With the motivation to have a better insight into the structure of this nonlinear transmission eigenvalue problem, in this paper we study a related Born transmission eigenvalue (BTE) problem. To be more precise, we note that for any fixed $k$ one can expand

\begin{equation} \label{n_eps expansion}
n_\epsilon(x,k) = 1 + \epsilon^d \left[ m(x) +o(1) \right],  
\end{equation}

\n where the error term $o(1)$ converges to zero, as $\epsilon \to 0$ pointwise in $x\neq 0$ (as well as in $L^p(B_2)$ for suitable $p$), and $m$ is given by

\begin{equation} \label{m}
m(x) = m(|x|) = \frac{(2+|x|)^{d-1}}{2^d |x|^{d-1}}.
\end{equation}

\n Broadly speaking, in the weak scattering regime, or when the contrast of an obstacle is small, i.e. $n(x) = 1 + \epsilon m(x)$, where $n$ denotes the refractive index of the scattering obstacle and $\epsilon \ll 1$ is a small parameter, the Born approximation to the scattered field $u^s$ is given by $\partial_\epsilon u^s |_{\epsilon=0}$ (also, the Born approximation replaces the far filed operator by an operator that depends linearly on $m$, cf. \cite{CPS07}). For a background on the Born approximation and inverse scattering theory in the Born regime we refer to \cite{mosks08, KKMS09, KR14, CCR16, kir17}. From the cloaking perspective, existence of real BTEs would imply better or higher order invisibility results at those wave numbers. Also, given the non-discreteness phenomenon in the interior transmission eigenvalue problem, one can wonder about non-discreteness in the corresponding Born problem. We address these questions and answer them negatively for the function $m$ given by \eqref{m}. We mention that it is a common assumption in the literature for the refractive index to be a bounded function, however, here $m \notin L^\infty(B_2)$. 

The fact that there are no real BTEs is a well-known consequence of positivity of $m$ \cite{CPS07, CCR16}. The existence of countably many BTEs is obtained by making use of the radial geometry and following \cite{CCR16}: separating the variables, BTEs are described as zeros of a family of analytic functions. Another important, but harder question is about the location of the transmission eigenvalues in the complex plane. Specifically, it is useful to obtain information about eigenvalue free regions. The so-called linear sampling method \cite{CK96, CC06, CMS21} provides a technique to obtain qualitative information about the location and the shape of the scattering obstacle, from the scattered far field data. For justification of the linear sampling method in the time domain, the questions of discreteness and location of eigenvalues (relative to the real axis) are essential. With this in mind, we show that BTEs stay away from the real axis, i.e. there exists a strip in the complex plane, parallel to and containing the real axis, which does not contain any Born transmission eigenvalues. This will allow one to use the Fourier-Laplace transform and consider the time-domain linear sampling method \cite{CMS21}. Moreover, we also conjecture that any strip parallel to the real axis contains at most finitely many BTEs. For simplicity of presentation we carry out the proof for the function $m$ given by \eqref{m}. However, our methods apply to general positive and radial functions $m$ for which $t^{d-1} m(t)$  extends to a holomorphic function around the interval $[0,2]$, where $t$ denotes the radial variable (cf. Remark~\ref{REM m properties}). 

Questions regarding the distribution of transmission eigenvalues in the complex plane, such as eigenvalue free regions were studied in \cite{LC12, CL13, CLM15, vod15, vod16, PG17}. However, there are not many works in this direction for BTEs. For example, in \cite{CCR16} among other things, the authors study the distribution of those BTEs that correspond to radially symmetric eigenfunctions. One of the consequences of \cite{PG17} is that (for the ball with constant refractive index) all transmission eigenvalues lie in a horizontal strip containing the real axis. Our work exhibits qualitative difference between the distributions of transmission eigenvalues (TEs) and those in the Born regime (BTEs), it shows somewhat contrary, or dual behavior of BTEs, namely that all of them lie outside of some horizontal strip around the real axis. Or, invoking the conjectured stronger statement, all (but possibly finitely many) BTEs lie outside of any horizontal strip. Our proof uses Olver's uniform asymptotic expansion of Bessel functions of large order in the complex plane in terms of Airy functions \cite{olver54, olverTHMB54, olver74}. We also use contour deformations and Mellin transform techniques \cite{wong01}.

\section{The main result} \label{SECT main result}
\setcounter{equation}{0}

The transmission eigenvalue (TE) problem corresponding to the cloaking scheme described in the introduction is anisotropic, however changing the variables in the transmitted field inside the cloak using the map $F_\epsilon$ and leaving the incident field unchanged, one can get rid of the anisotopy. As a result we arrive at the following equivalent TE problem in $\RR^d$ for $d=2$ or $d=3$ (cf. (4.1) in \cite{CHV23}):

\begin{equation} \label{trans}
\begin{cases}
\Delta u  + k^2 n_\epsilon(\cdot, k) u = 0, \qquad \qquad &\text{in} \ B_2 \backslash \overline{B}_\epsilon
\\
\Delta v + k^2 v = 0, &\text{in} \ B_2
\\
u = v,  &\text{on} \ \partial B_2
\\
\partial_\nu u = \partial_\nu v &\text{on} \ \partial B_2
\\
u = 0, &\text{on} \ \partial B_\epsilon,
\end{cases}
\end{equation}

\n where $n_\epsilon$ is given by \eqref{n_eps}. Assume the function $u$ from \eqref{trans} can be approximated by $u \approx u_0 + \epsilon^d u_1$, where $u_0, u_1$ are defined in $B_2$. When $v$ is of the same order as $u_0$, for $u_1$ we consider the following approximation (for scattering estimates from a small circular obstacle we refer to, e.g. \cite{HPV07})

\begin{equation} \label{trans w1}
\begin{cases}
\Delta u  + k^2 u = -k^2 m v, \qquad \qquad &\text{in} \ B_2
\\
\Delta v + k^2 v = 0, &\text{in} \ B_2
\\
u = \partial_\nu u = 0,  &\text{on} \ \partial B_2,
\end{cases}
\end{equation}

\n where $m$ is given by \eqref{m} and with abuse of notation we dropped the subscript from $u_1$ and still called this function $u$.

\begin{definition} \label{DEF}
\normalfont
$k \in \mathbb{C} \backslash \{0\}$ is called a Born transmission eigenvalue (BTE) if there exists a nontrivial solution $u,v$ to \eqref{trans w1}, such that $u \in H^2$ and $m v \in L^2$, where we suppress the domain $B_2$ from the notation of the function spaces.
\end{definition}

Some remarks are now in order:

\begin{remark} \mbox{} \label{REM discussion}

$\bullet$ Note that when we formally take limits in \eqref{trans}, the equation for $u_1$ holds in the punctured ball $B_2 \backslash \{0\}$. Apriori $u_1$ may have a singularity at the origin and it might be interesting to consider families of interior transmission problems where the equation for $u_1$ holds in the punctured ball and has a prescribed singularity at the origin. This work is not concerned with the mathematical justification of the limiting procedure leading to \eqref{trans w1}. We assume that $u_1$ is regular at the origin and its equation is satisfied in $B_2$. 

\vspace{.05in}

$\bullet$ Due to the singularity of $m$ there are no eigenvalues of \eqref{trans w1} in the radially symmetric case. Indeed, say $d=2$ and $v$ is radial, then up to a multiplicative constant $v(x) = J_0(k |x|)$, where $J_0$ is the Bessel function of order 0. As $J_0$ is regular near the origin we see that $mv \notin L^2$ due to the singularity of $m$. However, note that $\sqrt{m} v \in L^2$. Therefore, this issue can be remedied by requiring $u, v$ to lie in the weighted spaces $H^2_{0, \frac{1}{m}}$ and $L^2_m$, respectively (see also Remark~\ref{REM B_0} in Section~\ref{Discreteness}). Where the notation means that $\sqrt{m} v \in L^2$ and $\frac{1}{\sqrt{m}} \partial^\alpha u \in L^2$ for any multiindex $|\alpha| \leq 2$. The distribution of BTEs in the radially symmetric case was considered in \cite{CCR16} (for nonsingular $m$) and in this case BTEs are zeros of a single entire function. In this work BTEs are zeros of a countable family of entire functions and consequently analyzing their distribution in the complex plane is much more technical.  

\vspace{.05in}

$\bullet$ It would be interesting to consider transmission eigenvalue problems like \eqref{trans w1} in general domains with general singular weights $m$ and analyze their solutions in the weighted Sobolev spaces mentioned above. This will be a task for the future.

\end{remark}

The goal of this paper is to analyze the Born transmission eigenvalue problem \eqref{trans w1}. We point out that our core contribution is establishing part $(iv)$ of the next theorem.

\vspace{.05in}

\begin{theorem} \label{THM main}
Let $d=2$ or $d=3$, and $m$ be given by \eqref{m}. The following are true for the Born transmission eigenvalue problem \eqref{trans w1}:

\begin{enumerate}
\item[(i)] There are no real, or purely imaginary BTEs.

\item[(ii)] There are infinitely many BTEs in $\mathbb{C} \backslash (\RR \cup i\RR)$.

\item[(iii)] BTEs form a discrete set in $\mathbb{C}$ (i.e. a countable set with no limit points in $\mathbb{C}$).

\item[(iv)] There exists $c>0$, such that there are no BTEs in the strip $\{k\in \mathbb{C} : |\Im k| < c\}$.
\end{enumerate}

\end{theorem}

\begin{conjecture}
For any $C>0$, the strip $\{k\in \mathbb{C} : |\Im k| < C\}$ contains at most finitely many BTEs.
\end{conjecture}

\begin{remark} \label{REM m properties}
Our methods of analysis for part $(iv)$ apply to any positive and radially symmetric function $m=m(t)$, where $t = |x|$ denotes the radial variable, for which $t^{d-1} m(t)$ extends to a holomorphic function in some open set containing the interval $[0,2]$. In particular (also in view of Remark~\ref{REM B_0} in Section~\ref{Discreteness}), our approach can be adapted to classical Born transmission eigenvalue problems \eqref{trans w1} with non-singular (bounded) $m$.
\end{remark}

The starting point in proving the above theorem is to use the radial symmetry and separation of variables, as in \cite{CCR16} to characterize BTEs as zeros of a countable family of entire functions $\{\CB_n(k)\}_{n=1}^\infty$. These functions are defined as integrals of the Bessel function (or its spherical counterpart for $d=3$) of the first kind and order $n$, against an explicit function involving $m$ (cf. Lemma~\ref{LEM B_n intrep}). The first three items of Theorem~\ref{THM main} are proved using basic asymptotic approximations of Bessel functions. Item $(iv)$ is significantly harder to prove and relies on Olver's asymptotic formulas that exhibit the fine behavior of the Bessel functions. Let us now outline the main idea behind proving the item $(iv)$. We proceed by contradiction and assume that there exists a sequence $\{k_j\}_{j=1}^\infty \subset \mathbb{C}$ of distinct BTEs approaching the real line, i.e. $\Im k_j \to 0$, as $j \to \infty$. Clearly, $\CB_{n_j}(k_j) = 0$ for some integer $n_j$ and all $j$. Using standard approximations and estimates of Bessel functions we first show that for each fixed $n$, the function $\CB_n(k)$ cannot have a sequence of distinct zeros lying in any horizontal strip in the complex plane. This allows us to conclude that $n_j \to \infty$. The discreteness of BTEs (part $(iii)$ of Theorem~\ref{THM main}), on the other hand implies that $\Re k_j \to \infty$. We then find an explicit expression $E(n_j, k_j)$ such that, upon passing to a subsequence if necessary, which we do not relabel,

\begin{equation} \label{EB ineq}
\left| E(n_j, k_j) \CB_{n_j}(k_j) \right| \geq C
\end{equation}

\n for some $C>0$ and all $j$ large enough. This provides the desired contradiction. The behavior of $\CB_{n_j}(k_j)$, or the explicit expression of $E(n_j, k_j)$ depends on the growth of $\{n_j\}$ relative to $\{k_j\}$. There are three different regimes, depending on whether one of these sequences grows at a rate comparable to, much faster or slower than the other one. The most delicate regime is arguably when the two sequences grow at a comparable rate. Special attention is given to the case where $k_j/ n_j \to 1$. This delicate behavior is due to the so-called turning point of the Bessel's differential equation \cite{olver74} at the point 1. 

\section{Discreteness and existence of the Born transmission eigenvalues} \label{Discreteness}
\setcounter{equation}{0}

Here we prove items $(i)-(iii)$ of Theorem~\ref{THM main}. We start by describing BTEs as zeros of a family of entire functions. This can be achieved by using the radial symmetry and separating the variables as is done in Theorem 1 of \cite{CCR16}. We omit the details and just state the result, as the proof is completely analogous. Throughout this paper, $J_n$ and $j_n$ denote the Bessel and spherical Bessel functions, respectively, of the first kind and order $n$. Recall that

\begin{equation} \label{j_n via J_n}
j_n(z) = \sqrt{\frac{\pi}{2z}} J_{n+\frac{1}{2}}(z).
\end{equation}

\begin{lemma} \label{LEM B_n intrep}
For $n=1,2...$ introduce the functions

\begin{equation} \label{B_n}
\CB_n(k) = 
\begin{cases}
\displaystyle \int_0^1 f(t) J_n^2(k t) d t, \qquad  & $d=2$
\\[.2in]
\displaystyle \int_0^1 f(t) j_n^2(k t) d t, & $d=3$,
\end{cases}
\end{equation}

\n where

\begin{equation} \label{f}
f(t) = 
\begin{cases}
\displaystyle 4t m(2t) , \qquad  & $d=2$
\\[.2in]
\displaystyle 8t^2 m(2t), & $d=3$
\end{cases}
\quad =
\begin{cases}
\displaystyle t + 1 , \qquad  & $d=2$
\\[.2in]
\displaystyle (t+1)^2, & $d=3$.
\end{cases}
\end{equation}

\n The set of zeros

\begin{equation*}
\left\{ k \in \mathbb{C} \backslash \{0\}: \CB_n(k) = 0 \ \text{for some} \ n \right\}
\end{equation*}

\n coincides with the set of BTEs for \eqref{trans w1}.

\end{lemma}

\begin{remark} \label{REM B_0}
In the radially symmetric case $n=0$ and if we allow the eigenfunctions $u, v$ to lie in the weighted spaces $H^2_{0, \frac{1}{m}}$ and $L^2_m$, respectively (as discussed in Remark~\ref{REM discussion}), then the zeros of the function $\CB_0$ will also be eigenvalues. Therefore, throughout this paper we study the functions $\{\CB_n\}$ including the index $n=0$. 
\end{remark}

Next we collect some basic properties of the functions $\CB_n$.

\begin{lemma} \label{LEM B_n properties}
Let $n\geq 0$ be an integer and $\CB_n(k)$ be given by \eqref{B_n}. Then

\begin{enumerate}
\item[(i)] $\CB_n(k)$ is an entire function that has infinitely many zeros in $\CC$ and it has no real, or purely imaginary zeroes (apart possibly from $k=0$).
\item[(ii)] $\overline{\CB_n(k)} = \CB_n(-\overline{k})$.
\item[(iii)] $\CB_n \to 0$, as $n \to \infty$ uniformly on compact subsets of $\mathbb{C}$.
\item[(iv)] Let $\{k_j\}_{j=1}^\infty \subset \mathbb{C}$ be a sequence of distinct zeros of $\CB_n(k)$, for some fixed $n$, then $|\Im k_j| \to \infty$ as $j \to \infty$.
\end{enumerate}
\end{lemma}

\n To prove the existence of infinitely many zeros of $\CB_n$ and the item $(iv)$ of the above lemma, we borrow some ideas from Theorems 2 and 3 of \cite{CCR16}. The proof of Lemma~\ref{LEM B_n properties} is given in Appendix~\ref{SECT App B_n}. 

Note that the above two lemmas directly imply items $(i)$ and $(ii)$ of Theorem~\ref{THM main}. Now we turn to proving item $(iii)$ of Theorem~\ref{THM main}. 

\begin{lemma}[Discreteness] \mbox{}

\n Let $\CB_n(k)$ be given by \eqref{B_n}, then the set $\{k\in\CC: B_n(k)=0 \ \text{for some} \ n\geq 0\}$ is discrete in $\CC$, i.e. it is a countable set with no limit points in $\CC$.
\end{lemma}

\begin{proof}
We give the proof for $d=2$, as for $d=3$ the argument is analogous. Assume, for the sake of a contradiction that $k\in\CC$ is a limit point of the set of zeros of $\{\CB_n\}$, let $\{k_j\}$ be a sequence of distinct points such that $\CB_{n_j}(k_j) = 0$ for some $n_j \geq 0$, all $j\geq 1$ and $k_j \to k$. Let us assume that $k_j \neq 0$ for all $j$. If $\{n_j\}$ stays bounded, then $\{k_j\}$ is a sequence of zeros of finitely many $\{\CB_{n_j}\}$ and upon passing to a subsequence $\{k_j'\} \subset \{k_j\}$, for some fixed integer $n_0$ we have $\CB_{n_0}(k_j') = 0$ for all $j$. Since $\CB_{n_0}$ is an entire function we conclude that $\CB_{n_0} \equiv 0$, which is a contradiction. Thus, we may assume $n_j \to \infty$.

We will use the large order asymptotic expansion of the Bessel function (cf. 9.3.1 of \cite{handbook}): as $n \to \infty$

\begin{equation} \label{large n z fixed}
J_n(z) \sim \frac{1}{\sqrt{2\pi n}} \left( \frac{ez}{2n} \right)^n,
\end{equation}

\n which holds uniformly for $z \in \CC$ bounded. In our case $z=k_j t$ stays bounded uniformly for $j$ and $t\in[0,1]$. Consequently, as $j\to\infty$

\begin{equation*}
\CB_{n_j}(k_j) \sim \frac{1}{2\pi n_j} \left( \frac{e k_j}{2n_j} \right)^{2n_j} \int_0^1 t^{2n_j} f(t) dt.  
\end{equation*} 

\n Therefore, for $j$ large enough

\begin{equation*}
|\CB_{n_j}(k_j)| \geq \frac{1}{4\pi n_j} \left( \frac{e |k_j|}{2n_j} \right)^{2n_j} \int_0^1 t^{2n_j} f(t) dt,  
\end{equation*} 

\n which provides the desired contradiction, as the expression on the right hand side is positive: $f(t)>0$ for all $t\in(0,1)$. 
 
\end{proof}

\section{The Born transmission eigenvalues and the real line} \label{SECT Born and R}
\setcounter{equation}{0}

The goal of this section is to establish item $(iv)$ of Theorem~\ref{THM main}. Let us introduce

\begin{equation} \label{alpha beta}
\alpha(z) = \ln \left( \frac{1+\sqrt{1-z^2}}{z}\right) - \sqrt{1-z^2},
\qquad \qquad
\beta(z) = \sqrt{z^2-1}-\arccos(z^{-1}),
\end{equation}

\n where all the functions are defined by their principal branches. Let also

\begin{equation} \label{V_0 and V_1}
V_0 = \{z \in \CC: 0<\Re z < 1\}, \qquad \qquad V_1 = \{z \in \CC: \Re z > 1\}.
\end{equation}

\n We remark that $\alpha$ is analytic in $\{\Re z > 0\} \backslash [1,\infty)$, while $\beta$ in $\{\Re z > 0\} \backslash (0,1]$. These functions show up in the uniform asymptotic expansion of Bessel functions of large order as stated in the lemma below.

\begin{lemma} \label{LEM Bessel unif}
Let $\nu \in \RR$, $\epsilon \in (0,1)$ and assume the notation introduced above. As $\nu \to \infty$,

\begin{equation} \label{Bessel unif z<1}
J_\nu^2(\nu z) \sim \frac{1}{2\pi \nu \sqrt{1-z^2}} e^{-2 \nu \alpha(z)},
\end{equation}

\n uniformly for $z\in V_0$ with $|z|<1-\epsilon$. 

Let now $\{z_\nu\} \in V_1$ be a sequence with $|z_\nu|>1+\epsilon$ and

\begin{equation} \label{nu im z_nu}
\nu |\Im z_\nu| \leq c,
\end{equation}

\n for all $\nu$ large enough and some constant $c>0$ independent of $\nu$. Then, as $\nu \to \infty$

\begin{equation} \label{Bessel unif z>1}
J_\nu^2(\nu z_\nu) = \frac{2}{\pi \nu \sqrt{z_\nu^2-1}} \left[ \cos^2\left(\nu \beta(z_\nu) - \tfrac{\pi}{4}\right) + O \left( \frac{1}{\nu} \right)  \right],
\end{equation}

\n where the implicit constant in $O$ depends only on $\epsilon$ and $c$.

\end{lemma}

\n The above lemma directly follows from the uniform asymptotic formulas of Bessel functions due to Olver. Specifically we first use the formulas 9.3.35 - 9.3.42 of \cite{handbook} to expand $J_n$ in terms of the Airy function and its derivative (see also \eqref{Bessel uniform} in the Appendix, where this expansion is written down for future use). We then use the large argument asymptotics of the Airy function and its derivative (\cite{handbook} 10.4.59 and 10.4.61) in the sector away from the real negative semiaxis. This yields the formula \eqref{Bessel unif z<1}. To obtain \eqref{Bessel unif z>1} we use the asymptotics of Airy functions (\cite{handbook} 10.4.60 and 10.4.62) in the sector containing the real negative semiaxis. The error bound in \eqref{Bessel unif z>1} follows from Theorem B of \cite{olverTHMB54}.

\vspace{.05in}

With these preliminaries, let us start the proof of part $(iv)$ of Theorem~\Ref{THM main}. For the sake of contradiction, assume that there exists a sequence $\{k_j\}_{j=1}^\infty \subset \mathbb{C}$ of distinct BTEs, such that $\Im k_j \to 0$, as $j \to \infty$. 

\begin{notation}
Unless stated otherwise, for the remainder of this section (including all the subsections) the limit and asymptotic relations are understood as $j \to \infty$.
\end{notation}

\n For each $k_j$ there exists an integer $n_j \geq 0$ with $\CB_{n_j}(k_j) = 0$. If $\{n_j\}$ stays bounded, upon passing to a subsequence, which we do not relabel, $\CB_{n_0}(k_j) = 0$ for all $j \geq 1$ and some fixed integer $n_0$. But item $(iv)$ of Lemma~\Ref{LEM B_n properties} implies that $\{k_j\}$ cannot lie in any horizontal strip, which contradicts the boundedness of $\{\Im k_j\}$.

Thus, we may assume $n_j \to \infty$. Further, if $\{\Re k_j\}$ stays bounded, passing to a subsequence one concludes that BTEs have a limit point in $\mathbb{C}$, which cannot happen due to their discreteness as established in part $(iii)$ of Theorem~\Ref{THM main}. Therefore, $\{\Re k_j\}$ is unbounded and as BTEs are symmetric about the imaginary axis, we restrict our analysis to the right half-plane and assume that $\Re k_j \to +\infty$. To summarize we have:

\begin{equation} \label{limits}
n_j \to \infty,
\qquad \qquad
\Re k_j \to +\infty, \quad \text{and} \quad \Im k_j \to 0.
\end{equation}

\n As described in Section~\ref{SECT main result}, to achieve the desired contradiction our goal will be to establish a lower bound on $\CB_{n_j}(k_j)$ of the form \eqref{EB ineq}. The behavior of $\CB_{n_j}(k_j)$ is quite delicate and depends on the growth rate of the sequence $\{n_j\}$ relative to $\{k_j\}$. There are three main regimes:  

\begin{equation} \label{cases}
\begin{split}
\text{Case I:}& \quad n_j \gg |k_j|
\\[.1in]
\text{Case II:}& \quad n_j \ll |k_j|
\\[.1in]
\text{Case III:}& \quad k_j/n_j \to L, \ \text{for some} \ L>0 \ \text{and one of the following holds true:}
\\
\text{a.} & \quad L<1 
\\
\text{b.} & \quad L>1 
\\
\text{c.} & \quad L=1
\end{split}
\end{equation}

\n In each of these regimes the behavior of $\CB_{n_j}(k_j)$ is different. Further, in the subcase $L=1$, it even depends on the rate of convergence of $k_j/n_j$ to $L=1$. Before outlining the main ideas, let us first show that the above cases are exhaustive for our purposes. Indeed, our desired contradiction will be obtained if a lower bound of form \eqref{EB ineq} is established for some subsequences of $\{n_j\}$ and $\{k_j\}$. In other words, we can pass to subsequences when necessary. Now, if Case II is not satisfied, then $k_j/n_j$ has a bounded subsequence. Passing to a further subsequence, without relabeling it, we may assume $k_j/n_j \to L$, for some $L \in \CC$. Due to \eqref{limits}, $L \geq 0$. If $L=0$, then Case I holds, otherwise: Case III.   

Turning to the main ideas, we will present all the proofs when $d=2$, as $d=3$ is completely analogous due to the relation \eqref{j_n via J_n} and we will only state the corresponding results with additional comments whenever necessary. Note that we can rewrite

\begin{equation} \label{B_n with z_j(t)}
\CB_{n_j}(k_j) = \int_0^{1} f(t) J_{n_j}^2(n_j z_j(t)) dt, \qquad \qquad z_j(t)= \frac{k_j}{n_j} t.
\end{equation}

\n The idea is to use the asymptotic relations in Lemma~\ref{LEM Bessel unif} to replace $J_{n_j}^2(n_j z)$ with an appropriate expression. The behavior of $J_{n_j}^2(n_j z)$, however, is very delicate near the so-called turning point $z=1$. Lemma~\ref{LEM Bessel unif} already demonstrates different behaviors depending on whether $z=z_j(t)$ lies to the left, or to the right of the line $\{\Re z=1\}$ and stays away from $z=1$ (when it approaches $z=1$ the behavior is quite complicated and depends on the convergence rate of $z_j$ to $1$). Case I of \eqref{cases} is the easiest to analyze, as $z=z_j(t)$ always stays away from $z=1$ and lies to the left of $\{\Re z =1\}$, therefore the asymptotic relation \eqref{Bessel unif z<1} can be readily used. The other two cases require more work. We split the integral in \eqref{B_n with z_j(t)} into parts such that for each $t$ in the respective integration region $z=z_j(t)$ stays away from $z=1$ and either the relation \eqref{Bessel unif z<1} or \eqref{Bessel unif z>1} can be used. It is the integral, in which $z_j(t)$ can be arbitrarily close to $z=1$ that needs special care. In Case II we use Mellin transforms to treat this integral. In Case III, apart from using Mellin transforms, we also use the asymptotic expansion of $J_{n_j}^2(n_j z)$ in terms of the Airy function and its derivative. The behavior of $\CB_{n_j}(k_j)$ is then described in terms of the Airy function.

Case I is treated in Section~\ref{SECT case 2}, Case II in Section~\ref{SECT n<k} and Case III in Section~\ref{SECT n comparable k}.  

\subsection{Case I: the regime $n_j \gg |k_j|$} \label{SECT case 2}

\n In \eqref{B_n with z_j(t)} for $j$ large enough $|z_j(t)|<1/2$ uniformly for $t \in (0,1)$. Therefore, applying the first part of Lemma~\Ref{LEM Bessel unif} and conclude that

\begin{equation} \label{B asymp case2 first}
\CB_{n_j}(k_j) \sim \frac{1}{2\pi n_j} \int_0^{1} \frac{f(t)}{\sqrt{1-z_j(t)^2}} e^{-2n_j \alpha(z_j(t))} dt \sim \frac{1}{2\pi n_j} \int_0^{1} f(t) e^{-2n_j \alpha(z_j(t))} dt,
\end{equation}

\n where in the second step we used that $z_j(t) \to 0$ uniformly in $t$. The function $\alpha$ given by \eqref{alpha beta}, blows up near the origin, namely as $z \to 0$ from inside the right half-plane, $\alpha(z) \sim \ln \left( \frac{2}{z}\right) -1$. Therefore we introduce the function

\begin{equation} \label{alpha tilde}
\widetilde{\alpha}(z) = \alpha(z) +1 - \ln \frac{2}{z} = 1 - \sqrt{1-z^2} + \ln \frac{1+\sqrt{1-z^2}}{2}.
\end{equation}

\n It is straightforward to show that as $z \to 0$ from inside the right half-plane, $\widetilde{\alpha}(z) \sim \frac{z^2}{4}$. In terms of the function $\widetilde{\alpha}$, we can rewrite the asymptotics \eqref{B asymp case2 first} as

\begin{equation*}
\CB_{n_j}(k_j) \sim \frac{1}{2\pi n_j} \left( \frac{k_j e}{2 n_j} \right)^{2n_j} \int_0^{1} f(t) t^{2n_j} e^{-2n_j \widetilde{\alpha}(z_j(t))} dt \sim \frac{1}{\left(n_j!\right)^2} \left( \frac{k_j}{2}\right)^{2n_j} \int_0^{1} f(t) t^{2n_j} e^{-2n_j \widetilde{\alpha}(z_j(t))} dt ,
\end{equation*}

\n where we used the Stirling's formula in the last step. We will prove below that for any $z$ with $0<\Re z < 1$ and $|z|<C$ we have

\begin{equation} \label{alpha tilde bound}
\left| \Im \widetilde{\alpha}(z) \right| \leq C |\Im z|.
\end{equation}

\n Applying this inequality, we conclude that for $j$ large enough

\begin{equation*}
n_j \left|\Im \widetilde{\alpha}(z_j(t))\right| \leq t |\Im k_j| \leq |\Im k_j|,  
\end{equation*}

\n which converges to zero as $j \to \infty$ uniformly in $t$, due to our assumption on $\{k_j\}$ (cf. \eqref{limits}). Therefore, as $j \to \infty$

\begin{equation*}
 e^{-i 2n_j \Im \widetilde{\alpha}(z_j(t))} \to 1,
\end{equation*}

\n uniformly for $t \in (0,1)$. Consequently,

\begin{equation} \label{B case 2 positive}
\left(n_j!\right)^2 \left( \frac{2}{k_j} \right)^{2n_j} \CB_{n_j}(k_j) \sim  \int_0^{1} f(t) t^{2n_j} e^{-2n_j \Re \widetilde{\alpha}(z_j(t))} dt.
\end{equation}

\n The right hand side of the above asymptotic relation is real-valued and positive. This provides the desired contradiction. 

Thus, to finish the proof it remains to establish \eqref{alpha tilde bound}. Let $|z|<1-\epsilon$ for some small $\epsilon>0$ and $\Re z >0$. We have the representation

\begin{equation*}
\widetilde{\alpha}(z) = \int_0^z \frac{\xi d\xi}{1+\sqrt{1-\xi^2}} = \int_0^{\Re z} \frac{x dx}{1+\sqrt{1-x^2}} + \int_{\Re z}^z \frac{\xi d\xi}{1+\sqrt{1-\xi^2}},
\end{equation*}

\n where we chose the contour of integration to be the line segment joining $0$ to $\Re z$, followed by the one joining $\Re z$ to $z$. As the first integral in the above formula is real, it does not affect the imaginary part of $\widetilde{\alpha}$. For the second integral $|1+\sqrt{1-\xi^2}| \geq 1$, which implies the desired bound

\begin{equation*}
\left| \Im \widetilde{\alpha}(z) \right| \leq |z| \cdot |z-\Re z| < (1-\epsilon) |\Im z|.
\end{equation*}

\subsection{Case II: the regime $n_j \ll |k_j|$} \label{SECT n<k}
Due to the relation \eqref{j_n via J_n} between the Bessel functions $j_n$ and $J_n$, it is convenient to introduce the index notation

\begin{equation} \label{nu_j}
\nu_j= 
\begin{cases}
n_j, \qquad \qquad &d=2
\\
n_j+\frac{1}{2}, &d=3.
\end{cases}
\end{equation}

\n We proceed by letting $d=2$. The above notation helps to treat the case $d=3$ in parallel: when $d=3$, in what follows one just replaces $J_n$ with $j_n$. 

As already discussed in the paragraph below \eqref{B_n with z_j(t)}, we first split the integral defining $\CB_n$, into parts where $z_j(t)=k_j t/n_j$ stays away from $z=1$. To that end, let us fix a small parameter $\epsilon>0$ and split

\begin{equation} \label{B_n split n<k}
\CB_{n_j}(k_j) = \int_0^{(1-\epsilon) \frac{\nu_j}{|k_j|}} + \int_{(1-\epsilon) \frac{\nu_j}{|k_j|}}^{(1+\epsilon) \frac{\nu_j}{|k_j|}} + \int_{(1+\epsilon) \frac{\nu_j}{|k_j|}}^1,
\end{equation}

\n where the integrand $f(t) J_{n_j}^2(k_j t) dt$ was suppressed from the notation. Note that in the first and third integrals $z_j(t)$ stays away from $z=1$, while in the second one it can get arbitrarily close to $z=1$. To analyze the second integral we first change the variables

\begin{equation} \label{middle integral change of var}
\int_{(1-\epsilon) \frac{\nu_j}{|k_j|}}^{(1+\epsilon) \frac{\nu_j}{|k_j|}} f(t) J_{n_j}^2(k_j t) dt
= 
\frac{\nu_j}{k_j} \int_{(1-\epsilon)\frac{k_j}{|k_j|}}^{(1+\epsilon)\frac{k_j}{|k_j|}} f \left( \frac{\nu_j}{k_j} z \right) J_{n_j}^2 (\nu_j z) dz,
\end{equation}

\n where in the last integral we take the contour of integration to be the line segment connecting the two endpoints (in fact, it can be taken to be an arbitrary path, as the integrand is an entire function of $z$). Note that $k_j/|k_j| \to 1$, due to our assumptions on $k_j$ (cf. \eqref{limits}). Therefore, the contour of integration approaches to the interval $[1-\epsilon, 1+\epsilon]$ in the limit. Further, $k_j/|k_j| = e^{i \arg k_j}$ simply rotates this interval and so the integral in \eqref{middle integral change of var} is over the line segment $PQ$ as shown in Figure~\ref{FIG contour}. 

\begin{figure}[h]
\center
\captionsetup{width=.7\linewidth}
\includegraphics[scale=0.6]{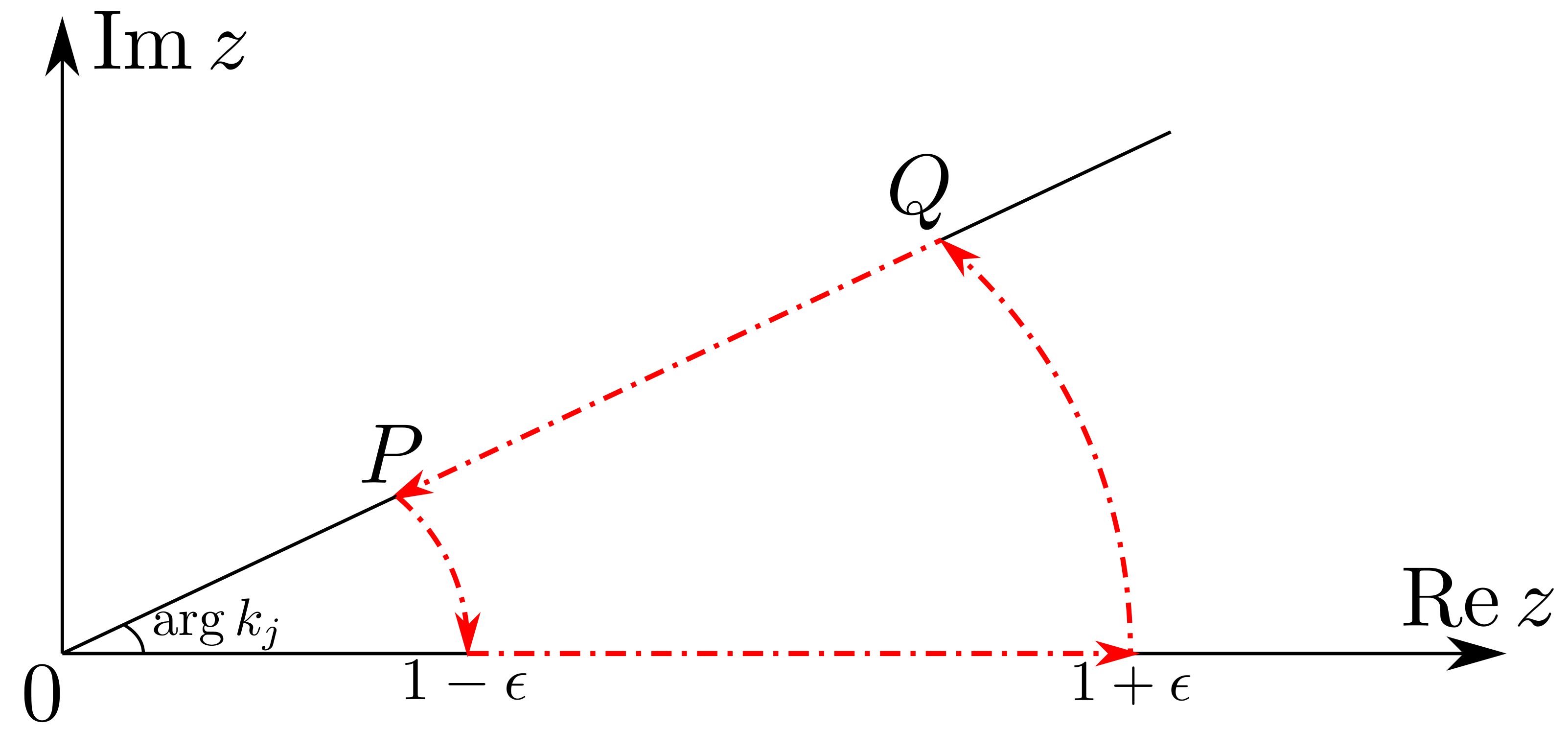}
\caption{The contour of integration.}
\label{FIG contour}
\end{figure}

\n We are going to analyze the integral \eqref{middle integral change of var} via Mellin transform, however this technique does not work when $z$ is complex (cf. Remark~\ref{REM Mellin complex var}). Therefore, we first need to deform the integral of the analytic function $z \mapsto f \left( \frac{\nu_j}{k_j} z \right) J_{n_j}^2 (\nu_j z)$ over $PQ$ to the integral over $[1-\epsilon, 1+\epsilon]$, taking into account the two integrals over circular arcs as shown in Figure~\ref{FIG contour}. The integral over the circular arc joining $P$ to $1-\epsilon$ will be grouped with the first integral in \eqref{B_n split n<k}, while the integral over the other circular arc will be grouped with the second integral in \eqref{B_n split n<k}. 
Putting things together we arrive at the following decomposition:

\begin{equation*}
\CB_{n_j}(k_j) = I_1 + I_2 + I_3,
\end{equation*}

\n where

\begin{equation} \label{I1 and I2}
\begin{split}
I_1 &= \int_0^{(1-\epsilon) \frac{\nu_j}{|k_j|}} f(t) J_{n_j}^2(k_jt) dt - i \frac{\nu_j}{k_j} (1-\epsilon) \int_0^{\arg k_j} f \left( \frac{\nu_j}{k_j} (1-\epsilon) e^{i \theta} \right) J_{n_j}^2 \left(\nu_j (1-\epsilon) e^{i \theta}\right) e^{i \theta} d \theta,
\\[.1in]
I_2 &= \int_{(1+\epsilon) \frac{\nu_j}{|k_j|}}^1 f(t) J_{n_j}^2(k_jt) dt + i\frac{\nu_j}{k_j} (1+\epsilon) \int_0^{\arg k_j} f \left( \frac{\nu_j}{k_j} (1+\epsilon) e^{i \theta} \right) J_{n_j}^2 \left(\nu_j (1+\epsilon) e^{i \theta}\right) e^{i \theta} d \theta
\end{split}
\end{equation}

\n and

\begin{equation} \label{I3}
I_3 = \frac{\nu_j}{k_j} \int_{1-\epsilon}^{1+\epsilon} f \left( \frac{\nu_j}{k_j} t \right) J_{n_j}^2 (\nu_j t) dt.
\end{equation}

\n Note that $I_l = I_l(n_j, k_j, \epsilon)$ for $l=1,2,3$, but these dependencies will be suppressed for the ease of notation. The advantage of above decomposition is that in $I_1, I_2$ we can use the large order asymptotics of the Bessel function from Lemma~\ref{LEM Bessel unif} as $z$ stays away from $1$. While in $I_3$, the integration variable is real and we can apply Mellin transform techniques. In the subsections below, we will show that

\begin{equation} \label{I1 I2 I3 n<k}
k_j I_1 \to 0, \qquad \qquad I_2 \sim \frac{\ln (k_j / n_j)}{\pi k_j}, \qquad \qquad k_j I_3 \to c_1, 
\end{equation}

\n where $c_1 = c_1(\epsilon)$ is some constant (cf. \eqref{kI3 asymp}). These asymptotic relations in particular imply

\begin{equation*}
\frac{\pi k_j}{\ln (k_j / n_j)} \CB_{n_j}(k_j) \sim 1,
\end{equation*}

\n which yields the desired inequality of the form \eqref{EB ineq} and contradicts the fact that $\{k_j\}$ are zeros of $\CB_{n_j}$.

\begin{remark} \label{REM case 1 d=3}
\normalfont
When $d=3$, the integrals $I_l$ for $l=1,2,3$ are given by the same formulas as above, except with $j_n$ in place of $J_n$. The asymptotic relations in this case take the form

\begin{equation*}
n_j k_j I_1 \to 0, \qquad \qquad n_j k_j I_2 \to \frac{1}{2} \arctan \frac{1}{\sqrt{\epsilon^2 + 2\epsilon}}, \qquad \qquad n_j k_j I_3 \to c_2, 
\end{equation*}

\n where $c_2=c_2(\epsilon)$ is a constant, such that $c_2(\epsilon) \to 0$ as $\epsilon \to 0$ (cf. \eqref{kI3 asymp 3d}). These readily imply that

\begin{equation*}
\lim_{j\to \infty} n_j k_j \CB_{n_j}(k_j) = \frac{1}{2} \arctan \frac{1}{\sqrt{\epsilon^2 + 2\epsilon}} + c(\epsilon).
\end{equation*}

\n As the left hand side of the last expression does not depend on $\epsilon$, we can take limits as $\epsilon \to 0$ and conclude

\begin{equation*}
\frac{4}{\pi} n_j k_j  \CB_{n_j}(k_j) \sim 1.
\end{equation*}

\end{remark}

\subsubsection{Analyzing $I_1$ via uniform asymptotics of the Bessel function} \label{SECT I1}

Let $I_1$ be given by \eqref{I1 and I2} and $d=2$. Here we show the first relation of \eqref{I1 I2 I3 n<k}, i.e. $k_j I_1 \to 0$. In fact, we will see below that this convergence to zero is at an exponential rate with respect to $n_j$. The formula of $I_1$ in \eqref{I1 and I2} contains two integral terms, the first of which we denote by $I_{11}$ and the second one by $I_{12}$, so that

$$I_1 = I_{11} + I_{12}.$$

\n We start by analyzing $I_{11}$. Note that it can be rewritten as

\begin{equation*}
I_{11} = \int_0^{(1-\epsilon) \frac{n_j}{|k_j|}} f(t) J_{n_j}^2(n_j z_j(t)) dt,
\qquad \qquad z_j(t) = \frac{k_j}{n_j} t.
\end{equation*}

\n For any $t$ in the integration interval

\begin{equation} \label{z_j < 1-eps}
|z_j(t)| \leq \frac{|k_j|}{n_j} t < 1-\epsilon 
\end{equation}

\n and $z_j(t)$ lies in the right half-plane, therefore the formula \eqref{Bessel unif z<1} of Lemma~\Ref{LEM Bessel unif} can be used to obtain the asymptotics, as $j \to \infty$, of the Bessel function in the integrand of $I_{11}$, uniformly for $t$ inside the integration interval. Hence, this asymptotic relation can be multiplied by $f(t)$ and integrated in $t$. This gives the asymptotic behavior of $I_{11}$. Multiplying the latter by $k_j$ and taking absolute values we conclude that, as $j\to\infty$

\begin{equation} \label{k_j I_11}
\left| k_j I_{11} \right| \sim \frac{|k_j|}{2\pi n_j} \left| \int_0^{(1-\epsilon) \frac{n_j}{|k_j|}} \frac{f(t)}{\sqrt{1-z_j^2(t)}} e^{-2n_j \alpha(z_j(t))} dt  \right| \leq \frac{C |k_j|}{\epsilon n_j} \int_0^{(1-\epsilon) \frac{n_j}{|k_j|}} e^{-2n_j \Re \alpha(z_j(t))} dt,
\end{equation}

\n where $\alpha$ is given by \eqref{alpha beta}, $C$ is an absolute constant and we used \eqref{z_j < 1-eps} to bound the square root term from below in the denominator of above integrand. Note that since $\{\Im k_j\}$ is bounded, due to \eqref{limits},

\begin{equation} \label{Im z_j to 0 n<k}
\Im z_j(t) = \frac{\Im k_j}{n_j} t \to 0, \qquad \text{as} \quad j \to \infty
\end{equation}

\n uniformly for $t$ inside the integration interval. Next we need the following lower bound on $\alpha$, which holds provided the imaginary part of its argument is sufficiently small (this is guaranteed by \eqref{Im z_j to 0 n<k}): there exists $j_0>0$ and $c=c_\epsilon>0$, such that

\begin{equation*}
\Re \alpha(z_j(t)) \geq c, \qquad j\geq j_0, \quad t \in \left(0,\tfrac{(1-\epsilon) n_j}{|k_j|}\right). 
\end{equation*}

\n This is a consequence of the explicit form of the function $\alpha$ and the proof can be found in part $(ii)$ of Lemma~\Ref{LEM alpha beta prop} in the Appendix (in fact, here we need the weaker version of the lower bound \eqref{alpha lower real}, where we drop the logarithm term. The stronger lower bound as formulated in \eqref{alpha lower real} is needed in the case $d=3$). Now, for large enough $j$, the right hand side of \eqref{k_j I_11} can be bounded by $C_\epsilon e^{-2n_jc_\epsilon}$, which converges to zero concluding the proof for $I_{11}$.

Let us now turn to $I_{12}$:

\begin{equation*}
I_{12} = - i\frac{n_j}{k_j} (1-\epsilon) \int_0^{\arg k_j} f \left( \frac{n_j}{k_j} z(\theta) \right) J_{n_j}^2 \left(n_j z(\theta)\right) e^{i \theta} d \theta,
\end{equation*}

\n where $z(\theta) = (1-\epsilon) e^{i \theta}$ does not depend on $j$. It has exactly the same properties as $z_j(t)$ above. The only difference from the above analysis is that the argument of $f$ now depends on $j$, but note that it converges to $0$ uniformly in $\theta$, hence $f$ can be replaced with $f(0)$ in the large $j$ asymptotics. Therefore,

\begin{equation*}
\left| k_j I_{12} \right| \sim \frac{(1-\epsilon) |f(0)|}{2\pi} \left| \int_0^{\arg k_j} \frac{1}{\sqrt{1-z^2(\theta)}} e^{-2n_j \alpha(z(\theta))} d\theta \right|.
\end{equation*}

\n As $\{\Im k_j\}$ is bounded and $\Re k_j \to +\infty$ we deduce that $\arg k_j \to 0$, so that $\Im z (\theta) = (1-\epsilon) \sin \theta$ is small for large $j$, uniformly for $\theta$ between $0$ and $\arg k_j$. Therefore, the above lower bound on $\Re \alpha$ can be used again, yielding the same conclusion: $k_j I_{12} \to 0$. Putting the estimates for $I_{11}$ and $I_{12}$ together, we conclude the proof.

\subsubsection{Analyzing $I_2$ via uniform asymptotics of the Bessel function} \label{SECT I2}

Let $I_2$ be given by \eqref{I1 and I2} and $d=2$. Here we prove the second relation of \eqref{I1 I2 I3 n<k}, i.e. $\pi k_j I_2 \sim \ln(k_j/n_j)$. The definition of $I_2$ in \eqref{I1 and I2} contains two integral terms, the first of which we denote by $I_{21}$ and the second one by $I_{22}$, so that

$$I_2 = I_{21} + I_{22}$$

\n The desired result will follow after showing

\begin{equation} \label{I_21 and I_22}
\pi k_j I_{21} \sim  \ln (k_j / n_j), \hspace{1.15in} k_j I_{22} \to 0.
\end{equation}

\begin{remark}
\normalfont

When $d=3$, the corresponding results read

\begin{equation*}
2n_j k_j I_{21} \to \arctan (\epsilon^2 + 2\epsilon)^{-\frac{1}{2}}, \qquad \qquad n_jk_j I_{22} \to 0.
\end{equation*}

\end{remark}

$\bullet$ We start from the term $I_{21}$:

\begin{equation*}
I_{21} = \int_{(1+\epsilon) \frac{n_j}{|k_j|}}^1 f(t) J_{n_j}^2(n_j z_j(t)) dt,
\qquad \qquad z_j(t) = \frac{k_j}{n_j} t.
\end{equation*}

\n Note that for any $t$ in the integration interval $|z_j(t)|>1+\epsilon$ and $z_j(t)$ lies in the right half-plane. Further, $n_j |\Im z_j(t)| \leq c$ for some $c>0$, as $\{\Im k_j\}$ is bounded. Therefore, we can apply the asymptotic formula \eqref{Bessel unif z>1} of Lemma~\Ref{LEM Bessel unif} for the Bessel function in the above integral, uniformly in $t$:

\begin{equation*}
I_{21} = \frac{2}{\pi n_j} \int_{(1+\epsilon) \frac{n_j}{|k_j|}}^1 \frac{f(t)}{\sqrt{z_j^2(t)-1}}  \left[ \cos^2\left(n_j \beta(z_j(t)) - \tfrac{\pi}{4}\right) + O\left( \tfrac{1}{n_j} \right)   \right] dt,
\end{equation*}

\n where $\beta$ is given by \eqref{alpha beta}. Letting $z = \frac{k_j}{n_j}t$ inside the integral gives

\begin{equation} \label{I_21 asymp}
I_{21} = \frac{2}{\pi k_j} \int_{(1+\epsilon) \frac{k_j}{|k_j|}}^{\frac{k_j}{n_j}} F_j(z) \left[ \cos^2\left(n_j \beta(z) - \tfrac{\pi}{4}\right) + O\left( \tfrac{1}{n_j} \right) \right] dz = \frac{1}{\pi k_j} \left( \CJ_1 + \CJ_2 \right),
\end{equation}

\n where the contour of integration is the line segment connecting the two endpoints in the above integral, and we set

\begin{equation*}
F_j(z) = \frac{f\left( \frac{n_j}{k_j} z \right)}{\sqrt{z^2-1}}
\end{equation*}

\n and

\begin{equation*}
\CJ_1 =  \int_{(1+\epsilon) \frac{k_j}{|k_j|}}^{\frac{k_j}{n_j}}  2 F_j(z) \cos^2\left(n_j \beta(z) - \tfrac{\pi}{4}\right) dz,
\qquad \qquad
\CJ_2 = O\left( \frac{1}{n_j} \right) \cdot  \int_{(1+\epsilon) \frac{k_j}{|k_j|}}^{\frac{k_j}{n_j}} F_j(z) dz.
\end{equation*}

\n We now show that the dominant term comes from $\CJ_1$, and $\CJ_2$ is of higher order. More precisely, as $j\to \infty$

\begin{equation*}
\CJ_1 \sim \ln \frac{k_j}{n_j},
\qquad \qquad
\CJ_2 = o \left( \ln \frac{k_j}{n_j} \right).
\end{equation*}

\n Indeed, as the argument of $f$ in $F_j(z)$ stays bounded, $f$ is also bounded. Hence,

\begin{equation*}
|\CJ_2| \leq \frac{C}{n_j} \int_{(1+\epsilon) \frac{k_j}{|k_j|}}^{\frac{k_j}{n_j}}  \frac{|dz|}{\sqrt{|z^2-1|}}  \leq \frac{C}{n_j}  \int_{(1+\epsilon) \frac{k_j}{|k_j|}}^{\frac{k_j}{n_j}}  \frac{|dz|}{|z|}  =  \frac{C}{n_j} \left( \ln \frac{|k_j|}{n_j} - \ln (1+\epsilon) \right),
\end{equation*}

\n which implies the desired estimate for $\CJ_2$. To analyze $\CJ_1$, we use the double angle formula to rewrite it as

\begin{equation*}
\CJ_1 = \int_{(1+\epsilon) \frac{k_j}{|k_j|}}^{\frac{k_j}{n_j}} F_j(z) dz + \int_{(1+\epsilon) \frac{k_j}{|k_j|}}^{\frac{k_j}{n_j}} F_j(z) \sin \left( 2 n_j \beta(z) \right) dz =: \CJ_{11} + \CJ_{12}.
\end{equation*}

\n Let us show that the dominant term is $\CJ_{11}$, and $\CJ_{12}$ is of lower order. Using the definition of $f$ from \eqref{f}:

\begin{equation} \label{J_11}
\begin{split}
\CJ_{11} =& \int_{(1+\epsilon) \frac{k_j}{|k_j|}}^{\frac{k_j}{n_j}} \frac{dz}{\sqrt{z^2-1}} + \frac{n_j}{k_j} \int_{(1+\epsilon) \frac{k_j}{|k_j|}}^{\frac{k_j}{n_j}} \frac{z}{\sqrt{z^2-1}} dz = 
\\
=& \int_{(1+\epsilon) \frac{k_j}{|k_j|}}^{\frac{k_j}{n_j}} \frac{dz}{z} 
+
\int_{(1+\epsilon) \frac{k_j}{|k_j|}}^{\frac{k_j}{n_j}} \frac{dz}{z\sqrt{z^2-1} \left(z + \sqrt{z^2-1} \right)} + \frac{n_j}{k_j} \int_{(1+\epsilon) \frac{k_j}{|k_j|}}^{\frac{k_j}{n_j}} \frac{z}{\sqrt{z^2-1}} dz,
\end{split}
\end{equation}

\n where in the last step we have added and subtracted $1/z$ in the first integral. It is now evident that

\begin{equation*}
\CJ_{11} \sim \ln (k_j/n_j).
\end{equation*}

\n Indeed, e.g. $|z + \sqrt{z^2-1}|$ is bounded from below and up to a constant, the modulus of the integrand in the second integral of \eqref{J_11} can be bounded by $1/|z|^2$. Similarly, the third integral in \eqref{J_11} can be estimated. 

Finally, we show that $\CJ_{12} \to 0$. To that end we integrate by parts and use the fact that $\beta'(z) = \sqrt{z^2-1}/z$:

\begin{equation*}
\begin{split}
-2n_j \CJ_{12} &= \int_{(1+\epsilon) \frac{k_j}{|k_j|}}^{\frac{k_j}{n_j}} \frac{z f\left( \frac{n_j}{k_j} z \right)}{z^2-1} \left[ \cos \left(2n_j \beta(z)\right) \right]' dz = \frac{\frac{k_j}{n_j} f(1) \cos \left(2n_j \beta(\frac{k_j}{n_j})\right)}{\left(\frac{k_j}{n_j} \right)^2-1} -
\\[.1in]
&-
\frac{(1+\epsilon)\frac{k_j}{|k_j|} f\left( (1+\epsilon)\frac{n_j}{|k_j|} \right) \cos (2n_j \beta((1+\epsilon)\frac{k_j}{|k_j|}))}{\left((1+\epsilon)\frac{k_j}{|k_j|} \right)^2-1}
+
\int_{(1+\epsilon) \frac{k_j}{|k_j|}}^{\frac{k_j}{n_j}} \frac{z^2 + 2\frac{n_j}{k_j}z + 1}{\left(z^2-1\right)^2} \cos \left(2n_j \beta(z)\right) dz
\end{split}
\end{equation*}

\n Next we need the following bound on the imaginary part of $\beta$ (see part $(i)$ of Lemma~\Ref{LEM alpha beta prop} in the Appendix):

\begin{equation*}
n_j |\Im \beta(z)| \leq c n_j |\Im z| \leq c,
\end{equation*}

\n for some $c>0$ and all $z$ lying on the integration line segment connecting $(1+\epsilon) \frac{k_j}{|k_j|}$ to $\frac{k_j}{n_j}$. Consequently, all the cosine terms in the above formula are bounded and hence the right hand side of the equation for $\CJ_{12}$ is bounded (for all $j$).

$\bullet$ Let us now turn to the analysis of $I_{22}$:

\begin{equation*}
I_{22} = i\frac{n_j}{k_j} (1+\epsilon) \int_0^{\arg k_j} f \left( \frac{n_j}{k_j} z(\theta) \right) J_{n_j}^2 \left(n_j z(\theta)\right) e^{i \theta} d \theta, \qquad \qquad z(\theta) = (1-\epsilon) e^{i \theta}.
\end{equation*}

\n Note that, for some $c>0$ and all $\theta$ in the integration interval

\begin{equation*}
n_j |\Im z(\theta)| \leq c n_j |\sin \theta| \leq c n_j |\theta| \leq c n_j |\arg k_j| = c n_j \arctan \frac{|\Im k_j|}{\Re k_j} \leq c \frac{n_j}{\Re k_j}. 
\end{equation*}

\n The last term in the above inequality tends to zero as $j \to \infty$, since $n_j \ll |k_j|$. However, we only need its boundedness to apply Lemma~\Ref{LEM Bessel unif}. Doing so, we replace the Bessel function with its asymptotic form \eqref{Bessel unif z>1}  (which holds uniformly in $\theta$) and use that the argument of $f$ converges to $0$ uniformly in $\theta$, so that $f$ can be replaced with $f(0)$. Thus, in the large $j$ asymptotics:

\begin{equation} \label{I_22 asymp}
\pi k_j I_{22} \sim 2i(1+\epsilon) f(0) \int_0^{\arg k_j} \frac{e^{i\theta}}{\sqrt{z^2(\theta)-1}} \left[ \cos^2\left(n_j \beta(z(\theta)) - \tfrac{\pi}{4}\right) + O\left( \tfrac{1}{n_j} \right) \right] d\theta. 
\end{equation}

\n As above, the cosine term is bounded uniformly in $j$ and $z$. Hence, the integrand in \eqref{I_22 asymp} can be bounded by a constant, and as $\arg k_j \to 0$ we conclude that $k_j I_{22} \to 0$.

\subsubsection{Analyzing $I_3$ via Mellin transform} \label{SECT I3}

We start by stating the main result of this section, which proves the third relation of \eqref{I1 I2 I3 n<k}. 

\begin{lemma} \label{LEM I3}
Let $I_3 = I_3(n_j, k_j, \epsilon)$ be given by \eqref{I3}, $c \in (0,1)$ and $\epsilon \in (0,1)$ then, as $j \to \infty$ 

\begin{equation} \label{kI3 asymp}
k_j I_3(n_j, k_j, \epsilon) \to \frac{1}{2\pi i} \int_{c+i\RR} 2^{z-1} \frac{\Gamma(1-z)}{\Gamma^2(1-\frac{z}{2})}  \left[ \frac{(1+\epsilon)^{1-z} - (1-\epsilon)^{1-z}}{1-z} \right] dz, 
\end{equation}

\n for $d=2$, where $\Gamma$ denotes the Gamma function. And for $d=3$ (cf. Remark~\ref{REM case 1 d=3})

\begin{equation} \label{kI3 asymp 3d}
n_j k_j I_3(n_j, k_j, \epsilon) \to - \frac{1}{4i} \int_{c+i\RR} 2^{z-1} \frac{\Gamma(1-z)}{\Gamma^2(1-\frac{z}{2})}  \left[ \frac{(1+\epsilon)^{-z} - (1-\epsilon)^{-z}}{z} \right] dz.
\end{equation}

\n In particular, upon applying the dominated convergence theorem

\begin{equation} \label{kI3 asymp eps 0}
\begin{split}
\lim_{\epsilon \to 0} \lim_{j \to \infty} k_j I_3(n_j, k_j, \epsilon) &= 0, \qquad \qquad d=2
\\
\lim_{\epsilon \to 0} \lim_{j \to \infty} n_j k_j I_3(n_j, k_j, \epsilon) &= 0, \qquad \qquad d=3.
\end{split}
\end{equation}

\end{lemma}

Let $\nu_j$ be given by \eqref{nu_j}. We start with the observation that $f \left( \frac{\nu_j}{k_j} t \right) \sim f(0) = 1$ uniformly in $t \in (1-\epsilon, 1+\epsilon)$, due to the assumption $n_j \ll |k_j|$. Therefore, the leading behavior of $I_3$ simplifies to 

\begin{equation} \label{I_3 leading term}
I_3 \sim \frac{\nu_j}{k_j} I_3^0,
\qquad \qquad I_3^0 = \int_{1-\epsilon}^{1+\epsilon} J_{n_j}^2 (\nu_j t) dt.
\end{equation}

\n The above definition of $I_3^0$ is for $d=2$. When $d=3$, $J_n$ must be replaced with $j_n$. To analyze the integral $I_3^0$ we use Mellin transform. The Mellin transform \cite{wong01} of a locally integrable function $\varphi$ on $(0,\infty)$, is defined by 

\begin{equation*}
\CM [\varphi] (z) = \int_0^\infty t^{z-1} \varphi(t) dt
\end{equation*}

\n for those values $z \in \mathbb{C}$ for which the above integral makes sense. Typically, the Mellin transform defines an analytic function in some vertical strip in the complex plane. Note that

\begin{equation} \label{I3 Mellin form}
\begin{split}
I_3^0 &= \int_0^\infty \varphi(t) J_{n_j}^2(n_j t) dt, \qquad \quad d=2
\\[.1in]
\frac{2k_j}{\pi} I_3^0 &= \int_0^\infty \varphi(t) J_{\nu_j}^2(\nu_j t) dt, \qquad \quad d=3,
\end{split}
\qquad \qquad
\varphi(t) = \frac{1}{t^{d-2}} \chi_{(1-\epsilon,1+\epsilon)}(t).
\end{equation}

\n So the question is reduced to considering the function

\begin{equation*}
I(\xi) = \int_0^\infty \varphi(t) J_{\nu}^2(\xi t) dt,
\qquad \qquad
\nu = \nu(n) = 
\begin{cases}
n, \qquad \qquad &d=2
\\
n+\frac{1}{2}, &d=3
\end{cases}
\end{equation*}

\n and analyzing the asymptotics of $I(\nu)$, where we suppressed the subscript $j$ from the notation. This convention will be followed throughout this section. Taking the Mellin transform of $I(\xi)$ with respect to $\xi$ and applying the inverse Mellin transform we obtain the representation

\begin{equation} \label{I Mellin}
I(\xi) = \frac{1}{2\pi i} \int_{c+i\RR} \xi^{-z} \CM [J_{\nu}^2](z) \CM[\varphi](1-z) dz,
\end{equation}

\n where it is assumed that the functions $\CM [J_{\nu}^2](z)$ and $\CM[\varphi](1-z)$ have a common strip of analyticity in the complex plane and $c+i\RR$ is a vertical line lying inside this strip. The representation \eqref{I Mellin} is also known as the Parseval formula \cite{wong01}.

\begin{remark}[Idea of the Mellin transform technique] \mbox{}
\normalfont

\n If $\CM [J_{\nu}^2](z)$ and $\CM[\varphi](1-z)$ can be analytically continued to meromorphic functions in a right half-plane and the contour of integration $c+i\RR$ can be shifted to the right, then the residues that are picked up in this process give the asymptotic expansion of $I(\xi)$ as $\xi \to \infty$ (for fixed $\nu$). This procedure, however, has limitations in our case as the integrand does not have sufficient decay to allow shifting the contour beyond the second pole. More importantly, the quantity of interest is $I(\nu)$, i.e. $\nu$-dependence occurs not only in the term $\CM [J_{\nu}^2]$, but also in $\nu^{-z}$. It turns out that if we shift the contour beyond the first pole, the residue does not give the dominant term and the shifted integral is of the same order as the residue.
\end{remark}

Direct calculation shows

\begin{equation*}
\CM[\varphi](1-z) = 
\begin{cases}
\displaystyle \frac{(1+\epsilon)^{1-z} - (1-\epsilon)^{1-z}}{1-z}, \qquad \qquad &d=2
\\[.2in]
\displaystyle - \frac{(1+\epsilon)^{-z} - (1-\epsilon)^{-z}}{z}, & d=3,
\end{cases}
\end{equation*}

\n which is an entire function of $z \in \mathbb{C}$. Next, for any real number $\nu \geq 0$ (cf. \cite{ober74})

\begin{equation} \label{J^2 Mellin}
\CM [J_\nu^2] (z) = 2^{z-1} \frac{\Gamma(\nu+\frac{z}{2}) \Gamma(1-z)}{\Gamma^2(1-\frac{z}{2}) \Gamma(1+\nu-\frac{z}{2})},
\qquad \qquad
-2\nu<\Re z < 1,
\end{equation}

\n which is a meromorphic function in $\mathbb{C}$ and \eqref{I Mellin} holds for any $c \in (-2\nu, 1)$. 

\begin{remark} \label{REM Mellin complex var}
\normalfont
Originally, we started with an integral (cf. \eqref{middle integral change of var}) where in $J_\nu^2(n \xi)$ the variable $\xi$ was complex. The Mellin transform approach cannot be applied in this case. Indeed, $\xi^{-z}$ is exponentially growing on $c+i\RR$ for complex $\xi$ and the integral \eqref{I Mellin} makes sense only if $\CM[\varphi](1-z)$ or $\CM [J_\nu^2] (z)$ is exponentially decaying. However, this is not the case as the well-known asymptotic formulas for the Gamma function imply that, as $|y| \to \infty$, uniformly for bounded $x$

\begin{equation} \label{MJ algebraic decay}
|\CM [J_\nu^2](x+iy)| = O \left( |y|^{x-\frac{3}{2}} \right).
\end{equation}

\n That is, $\CM [J_\nu^2]$ has only algebraic decay on the line $c+i\RR$ (for fixed $\nu$), which is due to the oscillatory behavior of $J_\nu^2$ for large arguments.
\end{remark}

\n With these preliminaries we are ready to prove Lemma~\Ref{LEM I3}, in the case $d=2$ (the case $d=3$ follows analogously).

\begin{proof}[Proof of Lemma~\Ref{LEM I3}]
In \eqref{I Mellin} let $c \in (0,1)$. Throughout the proof $C>0$ denotes an absolute constant that may change from one line to another.

\begin{equation*}
nI(n) = \frac{1}{2\pi i} \int_{c+i\RR} n^{1-z} \CM [J_{n}^2](z) \CM[\varphi](1-z)dz.
\end{equation*}

\n Our goal is to show that we can take limits as $n\to \infty$ in the above formula. We use the following asymptotic formula for the Gamma function (cf. \cite{CC14}):

\begin{equation} \label{Gamma asymp}
\Gamma(\zeta) \sim \sqrt{2\pi} e^{(\zeta - \frac{1}{2}) \ln (\zeta) - \zeta }
\end{equation}

\n uniformly, as $\zeta \to \infty$ in the sector $|\arg \zeta| \leq \pi - \delta$, where $\delta>0$ is a small number. Note that $n+\frac{z}{2}$ and $1+n-\frac{z}{2}$ have large modulus for large $n$, uniformly for $z \in c+i\RR$. Hence, for large enough $n$ we can use the above asymptotics to conclude that, as $n \to \infty$

\begin{equation*}
\begin{split}
n^{1-z} \frac{\Gamma(n+\frac{z}{2})}{\Gamma(1+n-\frac{z}{2})} \ \sim \ 
\ & e^{1-z} \exp\left\{ \left( n + \frac{z-1}{2} \right) \ln \left( 1 + \frac{z}{2n} \right) - \left( n - \frac{z-1}{2} \right) \ln \left( 1 + \frac{2-z}{2n} \right) \right\} =: 
\\
=:& h_n(z) 
\end{split}
\end{equation*}

\n and the asymptotics is uniform in $z \in c + i\RR$, therefore in view of \eqref{J^2 Mellin}

\begin{equation} \label{I31 F1}
nI(n) \sim  \int_{c+i\RR} F_n(z) dz,
\qquad \qquad
F_n(z) = \frac{1}{2\pi i} h_n(z) 2^{z-1} \frac{\Gamma(1-z)}{\Gamma^2(1-\frac{z}{2})}  \CM[\varphi](1-z).
\end{equation}

\n Since $h_n(z) \to 1$, as $n \to \infty$ for any fixed $z$, $F_n(z)$ has a pointwise limit. To conclude the proof it remains to show that we can put the limit, as $n \to \infty$, inside the integral of $F_n$ in \eqref{I31 F1}. This can be done via dominated convergence once we prove the bound

\begin{equation} \label{F1 bound}
|F_n(z)| \leq \frac{C}{|z|^{1.5}}, \qquad \qquad \forall z \in c+i\RR, \quad \forall n \ \text{large}.
\end{equation}

\n The remaining part of the proof is dedicated to establishing this bound. From now on let us always assume $n\geq 1$ and $z=c+iy$ with $y \in \RR$. Note that

\begin{equation} \label{Mphi bound}
|\CM[\varphi](1-z)| \leq \frac{C}{|z|}.
\end{equation}

\n Further, in view of \eqref{Gamma asymp}, as $|y| \to \infty$

\begin{equation} \label{Gamma ratio asymp}
\left| \frac{\Gamma(1-z)}{\Gamma^2(1-\frac{z}{2})} \right| \sim e \frac{|1-\frac{z}{2}|^{c-1}}{|1-z|^{c-\frac{1}{2}}} e^{p(z)},
\qquad \qquad
p(z) = y \left[\arg(1-z) - \arg\left(1-\tfrac{z}{2}\right)\right].
\end{equation}

\n Using this relation, there exists a constant $C>0$, such that

\begin{equation} \label{Gamma ratio bound}
\left| \frac{\Gamma(1-z)}{\Gamma^2(1-\frac{z}{2})} \right| \leq C \frac{|1-\frac{z}{2}|^{c-1}}{|1-z|^{c-\frac{1}{2}}} e^{p(z)} \leq  \frac{C}{|z|^{\frac{1}{2}}} e^{p(z)}
\end{equation}

\n  for all $z= c+iy$. Let us rewrite

\begin{equation} \label{h_n modulus}
|h_n(z)| = e^{1-c} |z|^{c-1} \left|\frac{1}{z}+\frac{1}{2n}\right|^{\frac{c-1}{2}} \left|\frac{1}{z}+\frac{2-z}{2nz}\right|^{\frac{c-1}{2}} \left| \frac{n+\tfrac{z}{2}}{n+\tfrac{2-z}{2}} \right|^n e^{q_n(z)},  
\end{equation}

\n where

\begin{equation*}
q_n(z) = -\frac{y}{2} \left[ \arg(1+\tfrac{z}{2n}) + \arg(1+\tfrac{2-z}{2n})  \right].
\end{equation*}

\n Using the basic estimates

\begin{equation*}
\left|\frac{1}{z}+\frac{1}{2n}\right| \geq \frac{1}{|z|}, \qquad \qquad
\left|\frac{1}{z}+\frac{2-z}{2nz}\right| \geq \frac{1}{|z|}
\end{equation*}

\n in \eqref{h_n modulus} we arrive at

\begin{equation} \label{h_n bound}
|h_n(z)| \leq C \left| \frac{n+\tfrac{z}{2}}{n+\tfrac{2-z}{2}} \right|^n e^{q_n(z)}.  
\end{equation}

\n Combining the bounds \eqref{Mphi bound}, \eqref{Gamma ratio bound} and \eqref{h_n bound} we deduce the bound

\begin{equation*}
|F_n(z)| \leq \frac{C}{|z|^{1.5}} \left| \frac{n+\tfrac{z}{2}}{n+\tfrac{2-z}{2}} \right|^n e^{p(z)+ q_n(z)}.
\end{equation*}

\n Let us show that $q_n(z)\leq 0$. Direct calculation gives

\begin{equation*}
\frac{d}{dn} q_n(z) = \frac{(1-c)4y^2 (2n+1)}{\left[ (2n+c)^2 + y^2 \right] \left[4n^2 + y^2 + 4n(2-c) + (2-c)^2 \right]} \geq 0,
\end{equation*}

\n therefore $q_n(z)$ is increasing in $n$ (for any fixed $z$), hence it is bounded by its limit as $n\to\infty$, which is equal to $0$. Next, $p(z) \leq 0$ for all $z =c+iy$. Finally,

\begin{equation*}
\left| \frac{n+\tfrac{z}{2}}{n+\tfrac{2-z}{2}} \right|^n = \left(1 - \frac{(1-c)(2n+1)}{(n+\tfrac{2-c}{2})^2 + \frac{y^2}{4}} \right)^\frac{n}{2} \leq 1.
\end{equation*}

\n Putting all these bounds together we obtain \eqref{F1 bound}.
\end{proof}

\subsection{Case III: the regime $k_j / n_j \to L$ with $L>0$} \label{SECT n comparable k}

\subsubsection{The case $L<1$} \label{SECT 3a}

Recall that

\begin{equation*}
\CB_{n_j}(k_j) = \int_0^{1} f(t) J_{n_j}^2(n_j z_j(t)) dt, \qquad \qquad z_j(t)= \frac{k_j}{n_j} t.
\end{equation*}

\n Since $L<1$, there exists a small $\epsilon>0$ such that for $j$ large enough

\begin{equation*}
|z_j(t)| \leq \frac{|k_j|}{n_j} < 1-\epsilon
\end{equation*}

\n uniformly for $t \in (0,1)$. Therefore, using \eqref{Bessel unif z<1} of Lemma~\Ref{LEM Bessel unif} we get

\begin{equation*}
\CB_{n_j}(k_j) \sim \frac{1}{2\pi n_j} \int_0^{1} \frac{f(t)}{\sqrt{1-z_j(t)^2}} e^{-2n_j \alpha(z_j(t))} dt.
\end{equation*}

\n Analogously to Section~\Ref{SECT case 2}, using that $z_j(t) \to Lt$ uniformly in $t$, we find  

\begin{equation*}
\left(n_j!\right)^2 \left( \frac{2}{k_j} \right)^{2n_j} \CB_{n_j}(k_j) \sim  \int_0^{1} \frac{f(t)}{\sqrt{1-L^2t^2}} t^{2n_j} e^{-2n_j \Re \widetilde{\alpha}(z_j(t))} dt,
\end{equation*}

\n where $\widetilde{\alpha}$ is defined by \eqref{alpha tilde}. The rest of the argument follows as in Section~\Ref{SECT case 2} and gives a lower bound of form \eqref{EB ineq} leading to a contradiction.

\subsubsection{The case $L>1$} \label{SECT 3c}

Let us fix a small $\epsilon>0$, such that $L>1+\epsilon$. Splitting the integral defining $\CB_{n_j}(k_j)$ similarly as in Section~\ref{SECT n<k} and deforming the contour of integration, we arrive at the following representation:

\begin{equation*}
\CB_{n_j}(k_j) = I_1 + I_2 + I_3,
\end{equation*}

\n where (we use the same letters $I_1,I_2,I_3$ for the integrals below, but these should not be confused with the integrals from Section~\ref{SECT n<k})

\begin{equation} \label{I1 and I2 case 3}
\begin{split}
I_1 &= \int_0^{(1-\epsilon) \frac{1}{L}} f(t) J_{n_j}^2(k_jt) dt + \frac{n_j}{k_j} \int_{(1-\epsilon) \frac{k_j}{n_j L}}^{1-\epsilon} f \left( \frac{n_j}{k_j} z \right) J_{n_j}^2 \left(n_j z\right) d z,
\\[.1in]
I_2 &= \int_{(1+\epsilon) \frac{1}{L}}^1 f(t) J_{n_j}^2(k_jt) dt + \frac{n_j}{k_j} \int_{1+\epsilon}^{(1+\epsilon) \frac{k_j}{n_j L}} f \left( \frac{n_j}{k_j} z \right) J_{n_j}^2 \left(n_j z\right) d z.
\end{split}
\end{equation}

\n In the $z$ integrals above we take the contour of integration to be the line segment connecting the two endpoints. Finally,

\begin{equation} \label{I3 case 3}
I_3 = \frac{\nu_j}{k_j} \int_{1-\epsilon}^{1+\epsilon} f \left( \frac{\nu_j}{k_j} t \right) J_{n_j}^2 (\nu_j t) dt \sim \frac{1}{L} \int_{1-\epsilon}^{1+\epsilon} f \left( \tfrac{t}{L}\right) J_{n_j}^2 (n_j t) dt .
\end{equation}

\n Similarly to Lemma~\Ref{LEM I3}, using Mellin transforms, it is straightforward to obtain the analogue of the equation \eqref{kI3 asymp eps 0} for the above integral $I_3$, i.e.

\begin{equation*}
\lim_{\epsilon \to 0}\lim_{j \to \infty} k_j I_3 = 0.
\end{equation*}

\n The term $I_1$ can be treated the same way as its analogue in Section~\Ref{SECT I1}, giving

\begin{equation*}
\lim_{j \to \infty} k_j I_1 = 0.
\end{equation*}

\n Finally, $I_2$ is analogous to the integral dealt with in Section~\Ref{SECT I2}. The asymptotic behavior, however, is different as now $k_j/n_j$ does not approach to zero and the term $k_j I_2$ does not have logarithmic singularity, instead we obtain

\begin{equation*}
\lim_{j \to \infty} k_j I_2 = \frac{1}{\pi} \int_{1+\epsilon}^L \frac{f\left( \tfrac{t}{L}\right)}{\sqrt{t^2-1}} dt.
\end{equation*}

\n Combining the above results we obtain 

\begin{equation*}
\lim_{j \to \infty} k_j \CB_{n_j}(k_j) = \frac{1}{\pi} \int_1^L \frac{f\left( \tfrac{t}{L}\right)}{\sqrt{t^2-1}} dt. 
\end{equation*}

\n This contradicts the fact that $k_j$ are zeros of $\CB_{n_j}$.

\begin{remark}
\normalfont
In the case $d=3$ the analogous result reads:

\begin{equation*}
\lim_{j \to \infty} n_j k_j \CB_{n_j}(k_j) = \frac{1}{2} \int_1^L \frac{f\left( \tfrac{t}{L}\right)}{t \sqrt{t^2-1}} dt .
\end{equation*}   
\end{remark}

\subsubsection{The case $L=1$} \label{SECT Case 3b}

\n This case is delicate, as the behavior of $\CB_{n_j}(k_j)$ depends on the rate of convergence of $k_j/n_j$ to 1, and whether it approaches 1 from the left, or from the right side of the line $\{\Re z = 1\}$. We start by changing the variables $z = \frac{k_j}{n_j} t$ to write

\begin{equation*}
\CB_{n_j}(k_j) = \frac{n_j}{k_j} \int_0^{\frac{k_j}{n_j}} f \left( \frac{n_j}{k_j} z \right) J_{n_j}^2 (n_j z) dz \sim \int_0^{\frac{k_j}{n_j}} f(z) J_{n_j}^2 (n_j z) dz, 
\end{equation*}

\n where we used that $f\left( \frac{n_j}{k_j} z \right) \sim f(z)$ as $j\to \infty$ uniformly for $z$ bounded and inside the right half-plane. Choosing the contour of integration to be the horizontal line segment $(0, \frac{\Re k_j}{n_j})$ followed by the vertical line segment connecting $\frac{\Re k_j}{n_j}$ to $\frac{k_j}{n_j}$ and changing the variables in the latter integral we obtain

\begin{equation*}
\CB_{n_j}(k_j) \sim \CJ_1 + \CJ_2,
\end{equation*}

\n where

\begin{equation*}
\CJ_1 = \int_0^{\frac{\Re k_j}{n_j}} f(z) J_{n_j}^2 (n_j z) dz,
\qquad\qquad
\CJ_2 = \frac{i}{n_j} \int_0^{\Im k_j} f\left( \frac{\Re k_j}{n_j} + i\frac{t}{n_j} \right) J_{n_j}^2 (\Re k_j + i t) dt.
\end{equation*}

\n Using that $f$ appearing in $\CJ_2$ is asymptotically equivalent to $f(1)=2$ uniformly in $t$ we find

\begin{equation} \label{J_2 case 3 asymp}
\CJ_2 \sim \frac{2i}{n_j} \int_0^{\Im k_j} J_{n_j}^2 (\Re k_j + i t) dt. 
\end{equation}

\n We remark that the Mellin transform approach does not yield the leading asymptotic behavior of $\CJ_1$. Indeed, the analogue of Lemma~\Ref{LEM I3} applied to this integral shows that

\begin{equation*} 
n_j \CJ_1 \to \frac{1}{2\pi i} \int_{c+i\RR} 2^{z-1} \frac{\Gamma(1-z)}{\Gamma^2(1-\frac{z}{2})}  \left[ \frac{1}{2-z} + \frac{1}{1-z}  \right] dz,
\end{equation*}

\n for $c\in(0,1)$. However, the right hand side of the above limit is 0, unlike the right hand sides of \eqref{kI3 asymp} and \eqref{kI3 asymp 3d}. Indeed, the integrand is an analytic function in the half-plane $\{\Re z < 1\}$ and converges to zero as $c \to -\infty$. The dominated convergence can be used to shift the contour $c+i\RR$ to $c \to -\infty$, implying that the integral is 0. Thus, we conclude that $n_j \CJ_1 \to 0$, which does not capture the leading behavior of $\CJ_1$.

The behavior of $\CJ_1$ and $\CJ_2$, in fact depends on the behavior of the sequence

\begin{equation*}
R_j := n_j^\frac{2}{3} \left( 1 - \frac{\Re k_j}{n_j} \right).
\end{equation*}

\vspace{.1in}
\n \textbf{Assume that} \ $\bm{\{R_j\}}$ \ \textbf{is bounded}, upon passing to a subsequence, which we do not relabel

\begin{equation} \label{R_j to R_infty}
R_j \to R_\infty,
\end{equation}

\n for some $R_\infty \in \RR$. Let us show that, with $\Ai$ denoting the Airy function,

\begin{equation} \label{B case 3b}
\CB_{n_j}(k_j) \sim n_j^{-\frac{4}{3}} \CJ_\infty,
\qquad \qquad
\CJ_\infty := \sqrt[3]{2} f(1) \int_{\sqrt[3]{2} R_\infty}^\infty \Ai^2(x) dx.
\end{equation}

\n This provides the desired contradiction.

\begin{remark}
\normalfont

When $d=3$, the analogous result reads

\begin{equation*}
 \CB_{n_j}(k_j) \sim \frac{\pi}{2} n_j^{-\frac{7}{3}} \CJ_\infty. 
\end{equation*}

\end{remark}

\vspace{.1in}

The conclusion \eqref{B case 3b} will follow after establishing

\begin{equation*}
n_j^\frac{4}{3} \CJ_1 \longrightarrow \CJ_\infty,
\qquad \qquad
n_j^\frac{4}{3} \CJ_2 \longrightarrow 0.
\end{equation*}

\n We start from the second assertion. In view of \eqref{J_2 case 3 asymp}, it is enough to prove that

\begin{equation} \label{J_2 n_j^1/3 to 0 case 3}
n_j^\frac{1}{3} \int_0^{\Im k_j} J_{n_j}^2 \left(n_j z_j(t) \right) dt \longrightarrow 0,
\qquad \qquad
z_j(t)=\frac{\Re k_j}{n_j} + i \frac{t}{n_j}.
\end{equation}

\n Note that $z_j(t) \to 1$ uniformly in $t$. All the asymptotics below are uniform in $t$ and we will suppress the $t$-dependence from the notation. Using the uniform asymptotics of the Bessel function \eqref{Bessel uniform} we conclude that, as $j\to \infty$

\begin{equation} \label{Bessel asymp case 3}
J_{n_j}^2(n_j z_j) =
\left( \frac{4 \zeta_j}{1-z_j^2} \right)^{\frac{1}{2}} n_j^{-\frac{2}{3}} \left[ \Ai(n_j^{\frac{2}{3}} \zeta_j) \left\{1 + O\left( n_j^{-2} \right) \right\} + n_j^{-\frac{4}{3}} b_0(\zeta_j) \Ai'(n_j^{\frac{2}{3}} \zeta_j) \left\{1 + O\left( n_j^{-2} \right) \right\} \right]^2,
\end{equation}

\n where $\zeta_j = \zeta(z_j)$ and $b_0$ are defined in Appendix~\ref{SECT App alpha beta}. Definition of $\zeta$ shows that $\zeta_j \to 0$, as $z_j \to 1$. The asymptotic relation $\zeta_j \sim \sqrt[3]{2} (1-z_j)$ (cf. \eqref{zeta asymp 1}) implies that

\begin{equation*}
n_j^\frac{2}{3} \zeta_j \sim \sqrt[3]{2} \left( R_j - it n_j^{-\frac{1}{3}} \right). 
\end{equation*}

\n Note that the first factor in \eqref{Bessel asymp case 3} stays bounded, and so does the term $b_0(\zeta_j)$ (cf. \eqref{b_0 at 0}). Moreover, as $\{R_j\}$ is bounded, so is the argument of the Airy function and its derivative in \eqref{Bessel asymp case 3}. Therefore, there exists a constant $C>0$ such that for all $j$ large enough and all $t\in (0, \Im k_j)$ we have 

\begin{equation*}
\left| J_{n_j}^2 (n_jz_j(t)) \right| \leq C n_j^{-\frac{2}{3}},
\end{equation*}

\n which implies the desired formula \eqref{J_2 n_j^1/3 to 0 case 3}.

Let us consider now the integral $\CJ_1$. Note that the variable $z$ is real and $z\in \left(0, \frac{\Re k_j}{n_j} \right)$. We can again use \eqref{Bessel asymp case 3} with $z, \zeta=\zeta(z)$ in place of $z_j, \zeta_j$. Observe that this expansion implies that 

\begin{equation*}
n_j^{\frac{4}{3}} \CJ_1 - \CJ_1^0 \longrightarrow 0,
\end{equation*}

\n where

\begin{equation*}
\CJ_1^0 = n_j^{\frac{2}{3}}  \int_0^{\frac{\Re k_j}{n_j}} f(z) \left( \frac{4 \zeta}{1-z^2} \right)^{\frac{1}{2}}  \Ai^2(n_j^{\frac{2}{3}} \zeta) dz.
\end{equation*}

\n Indeed, once we open up the square in \eqref{Bessel asymp case 3}, multiply the result by $n_j^{\frac{4}{3}}$ and integrate in $z\in \left(0, \frac{\Re k_j}{n_j} \right)$, all the terms converge to zero, apart from the first term, which is precisely the integral $\CJ_1^0$. Further explanation is needed. To establish this convergence, we need to use the dominated convergence theorem. Note that $z\in (0,1^+)$, where $1^+$ denotes a number slightly larger than 1, therefore $\zeta \in (0^-,\infty)$ and $\zeta(z)$ blows up near $z=0$ (cf. \eqref{zeta x}), so that the first factor of \eqref{Bessel asymp case 3} is singular, however the singularity is integrable:

\begin{equation*} 
\zeta(z) \sim \left( \frac{3}{2} \ln \left(\frac{2}{z}\right) - \frac{3}{2}\right)^\frac{2}{3},
\qquad \qquad z \to 0^+.
\end{equation*}

\n Therefore, we can concentrate on the remaining factors, for example the term $n_j^{\frac{2}{3}} \Ai^2(n_j^{\frac{2}{3}} \zeta) O\left( \tfrac{1}{n_j^4} \right)$ approaches to zero uniformly in $z$, as the Airy function is bounded on $\RR$. The only nontrivial term that remains to analyze is

\begin{equation} \label{Ai' term}
\left[ n_j^{-1} b_0(\zeta) \Ai'(n_j^{\frac{2}{3}} \zeta) \right]^2.
\end{equation}

\n The well-known asymptotic relations (\cite{handbook} 10.4.61 and 10.4.62) for the derivative of the Airy function imply that for some $C>0$

\begin{equation*}
|\Ai'(\xi)| \leq C(1+|\xi|^\frac{1}{4}),
\qquad \qquad \xi \in \RR.
\end{equation*}

\n Applying this bound we conclude that \eqref{Ai' term} converges to zero uniformly in $z$, as $b_0(\zeta)$ and $b_0(\zeta) |\zeta|^\frac{1}{4}$ are bounded functions for $z\in(0,1^+)$, i.e. for $\zeta \in (0^-, \infty)$.

So it remains to analyze $\CJ_1^0$ and prove that it has a limit. Let us change the variables $x = n_j^{\frac{2}{3}} \zeta(z)$ in $\CJ_1^0$. Then we have the representation

\begin{equation} \label{J_1^infty change var}
\CJ_1^0 = - \int_{n_j^{\frac{2}{3}} \zeta\left(\frac{\Re k_j}{n_j}\right)}^\infty f \left( \zeta^{-1}(n_j^{-\frac{2}{3}} x) \right) \left[ \frac{4 n_j^{-\frac{2}{3}} x}{1-\left(\zeta^{-1}(n_j^{-\frac{2}{3}} x)\right)^2} \right]^\frac{1}{2} \left( \zeta^{-1} \right)' (n_j^{-\frac{2}{3}} x) \Ai^2(x) dx.
\end{equation}

\n First, note that in view of the asymptotics of $\zeta(z)$ near $z=1$ \eqref{zeta asymp 1} we see that the lower bound in the above integral converges:

\begin{equation*}
n_j^{\frac{2}{3}} \zeta\left(\frac{\Re k_j}{n_j}\right) \sim \sqrt[3]{2} R_j \longrightarrow \sqrt[3]{2} R_\infty, 
\end{equation*}

\n due to our assumption \eqref{R_j to R_infty}. The integrand also converges pointwise: for any fixed $x$, when we let $j\to \infty$, we use that $\zeta^{-1}(0) = 1$ and $(\zeta^{-1})'(0) = 1/ \zeta'(1) = -2^{-\frac{1}{3}}$. Further, the limit of the fraction inside square brackets can be found from the relation

\begin{equation*}
\frac{4 \zeta(z)}{1-z^2} \sim \frac{4\sqrt[3]{2} (1-z)}{1-z^2} \sim 2\sqrt[3]{2},
\qquad \qquad z\to1.
\end{equation*}

\n It remains to show that the dominated convergence can be applied, which follows as the Airy function decays exponentially near infinity, $f$ is bounded, and the product of the second and third factors of the integrand in \eqref{J_1^infty change var} is also bounded. Indeed, the latter statement follows as the function

\begin{equation*}
\left[ \frac{4 \zeta(z)}{1-z^2} \right]^\frac{1}{2} \cdot \frac{1}{\zeta'(z)}
\end{equation*}

\n is bounded for $z \in (0,1^+)$. The limit at $z=1$ has been already discussed, and one can show that the above function has limit $0$ near $z=0$. Thus, we can take limits as $j\to \infty$ in \eqref{J_1^infty change var} and conclude the proof:

\begin{equation*}
\CJ_1^0 \longrightarrow \sqrt[3]{2} f(1) \int_{\sqrt[3]{2} R_\infty}^\infty \Ai^2(x) dx= \CJ_\infty.
\end{equation*}

\vspace{.1in}
\n \textbf{Assume now that} \ $\bm{\{R_j\}}$ \ \textbf{is unbounded}. Upon passing to a subsequence, which we do not relabel, there are two cases to consider:

\begin{equation*}
(i) \ \frac{\Re k_j}{n_j}>1 \quad \text{and} \ R_j \to -\infty,
\qquad \qquad
(ii) \ \frac{\Re k_j}{n_j}<1 \quad \text{and} \ R_j \to +\infty.
\end{equation*}

\n Indeed, note that if there is no subsequence for which $\Re k_j/n_j$ is always less than 1, or always larger than 1, then it must be equal to 1 (eventually), which implies $R_j=0$, i.e. $\{R_j\}$ is bounded and this was already analyzed above.

$(i)$ In this case the decay of $\CB_{n_j}(k_j)$ is slower than $n_j^{-\frac{4}{3}}$. For our purposes, it is enough to prove that there exists a constant $C>0$, such that for large enough $j$

\begin{equation*}
n_j^{-\frac{4}{3}} |\CB_{n_j}(k_j)| \geq C.
\end{equation*}

\n In fact, the above quantity goes to infinity as $j\to \infty$. The same analysis presented above applies and gives

\begin{equation*}
n_j^{\frac{4}{3}} \CJ_1 - \CJ_1^0 \longrightarrow 0,
\qquad \qquad
n_j^{\frac{4}{3}} \CJ_2 \longrightarrow 0,
\end{equation*}

\n which shows that for some $C>0$ and $j$ large enough

\begin{equation*}
n_j^{-\frac{4}{3}} |\CB_{n_j}(k_j)| \geq C \CJ_1^0. 
\end{equation*}

\n For $\CJ_1^0$ the same formula \eqref{J_1^infty change var} holds (note that this integral is positive, as $\zeta$ is a decreasing function). But now the lower limit of integration in \eqref{J_1^infty change var} goes to $-\infty$ as $j \to \infty$. It is trivial to get a lower bound for $\CJ_1^0$: truncate the integral to start, say, from $0$, then take the limit $j\to \infty$ inside the integral 

\begin{equation*}
\liminf_{j \to \infty} \CJ_1^0 \geq \sqrt[3]{2} f(1) \int_0^\infty \Ai^2(x) dx.
\end{equation*}

\vspace{.1in}

$(ii)$ In this case the decay of $\CB_{n_j}(k_j)$ is much faster than $n_j^{-\frac{4}{3}}$. We start from the definition

\begin{equation*}
\CB_{n_j}(k_j) = \int_0^1 f (t) J_{n_j}^2 (n_j z_j(t)) dt,
\qquad \qquad
z_j(t) = \frac{k_j}{n_j}t.
\end{equation*}

\n Let us suppress the $t$-dependence from the notation of $z_j$. Our goal is to use the uniform asymptotics of Bessel function \eqref{Bessel uniform}. Let $\alpha(z)$ be defined as in \eqref{alpha beta} and $\zeta(z)$ be given by \eqref{zeta z<1}. Note that, as $j \to \infty$

\begin{equation} \label{n alpha infty}
n_j \alpha(z_j) \to \infty
\end{equation}

\n uniformly in $t \in (0,1)$. Indeed, we just need to check this when $z_j$ is close to $1$, i.e. when $t$ is close to $1$ as this is the only point where $\alpha(z_j)$ becomes zero. But the relation

\begin{equation*} 
\alpha(z) = \frac{\sqrt[3]{2}}{3} (1-z)^\frac{3}{2} + O\left( (1-z)^\frac{5}{2}\right),
\end{equation*}

\n as $z\to 1$ with $\Re z<1$ shows that near $t=1$, the quantity $n_j |\alpha(z_j)|$ can be bounded from below by a constant multiple of $R_j$, which by assumption goes to infinity. The formula \eqref{n alpha infty} immediately implies that the argument of the Airy function in \eqref{Bessel uniform}: $n_j^{\frac{2}{3}} \zeta(z_j) \to \infty$ uniformly in $t$. In particular, the large argument asymptotics of the Airy function (\cite{handbook} 10.4.60 - 10.4.62) can be used to arrive at the formula

\begin{equation*}
J_{n_j}^2(n_j z_j) \sim \frac{1}{2\pi n_j \sqrt{1-z_j^2}} e^{-2 n_j \alpha(z_j)},
\end{equation*}

\n which again holds uniformly in $t \in (0,1)$. Thus,

\begin{equation*}
\CB_{n_j}(k_j) \sim \frac{1}{2\pi n_j}  \int_0^1 \frac{f(t)}{\sqrt{1-z_j^2}} e^{-2 n_j \alpha(z_j)} dt.
\end{equation*}

\n The rest of the argument of obtaining a lower bound on $\CB_{n_j}(k_j)$ is completely analogous to Section~\ref{SECT case 2}.

\section*{Acknowledgments} 
The author would like to thank F. Cakoni for suggesting the problem under consideration and to M. Vogelius and F. Cakoni for many fruitful conversations.

\appendix

\section{Appendix} \label{SECT appendix}
\setcounter{equation}{0} 
\subsection{Some properties of $\CB_n$} \label{SECT App B_n}

Here we give the proof of Lemma~\ref{LEM B_n properties}, for the case $d=2$. The case $d=3$ is completely analogous, due to the relation \eqref{j_n via J_n}.

\vspace{.05in}

$(i)$ $\CB_n(k)$ given by \eqref{B_n} is an entire function, since so is $J_n^2$ and we can differentiate inside the integral using the dominated convergence theorem. To show that $\CB_n$ has infinitely many zeros we are going to use the Hadamard factorization theorem \cite{conway73}. Let us start by showing that the order of the entire function $\CB_n$ is at most $1$. For any $z \in \mathbb{C}$ and $\nu \geq 0$, we have the bound (\cite{handbook} 9.1.62)

\begin{equation*}
|J_\nu(z)| \leq \frac{|z|^\nu e^{|z|}}{2^\nu \Gamma(\nu + 1)},
\end{equation*}

\n which implies that there exists a constant $c_n>0$ such that for $t \in (0,1)$ and $k \in \mathbb{C}$

\begin{equation*}
|J_n(kt)| \leq c_n |k|^n e^{|k|}.
\end{equation*}

\n Consequently, $|\CB_n(k)| \leq c_n |k|^{2n}  e^{2|k|} \leq e^{|k|^{1+\delta}}$, for any $\delta>0$ provided $|k|$ is large enough. The last estimate implies that the order of $\CB_n$ is at most 1. 

Suppose now that $\CB_n(k)$ has at most finitely many zeros. Let $k_1,...,k_N \in \mathbb{C} \backslash \{0\}$ be its zeros listed counting their multiplicities, then by Hadamard's factorization theorem

\begin{equation} \label{B_n Hadamard}
\CB_n(k) = k^s e^{ak+b} \prod_{j=1}^N \left( 1-\frac{k}{k_j} \right) 
\end{equation}

\n for some $s \in \{0,1,...\}$ and $a,b \in \mathbb{C}$ (if $\CB_n$ has no zeros then the above product is replaced by 1). Let us show that 

\begin{equation} \label{B_n at infinity}
\lim_{\RR \ni k \to \pm \infty} \CB_n(k) = 0.
\end{equation}

\n This will immediately contradict to the representation \eqref{B_n Hadamard}. The large argument asymptotics of the Bessel function (cf. \eqref{J_nu large z} below) implies that $J_n(kt) \to 0$, as $k \to \pm \infty$ along the real axis, for any fixed $t \in (0,1)$. This, along with the estimate $|J_n(x)| \leq 1$ for $x\in \RR$, allows us to apply the dominated convergence in the definition of $\CB_n$ \eqref{B_n} and conclude \eqref{B_n at infinity}.

It remains to show that $\CB_n$ cannot have zeros on the real and imaginary axes excluding the origin. First, $\CB_n(k) \neq 0$ for $k \neq 0$ real, as the integrand in \eqref{B_n} is a nonnegative function. Next, the Poisson representation formula (\cite{handbook} 9.1.20) for Bessel functions implies that $i^{-2\nu} J_\nu^2 \geq 0$ on $i\RR$ for any $\nu \geq 0$, i.e. up to a complex constant, $J_n^2$ is nonnegative on $i\RR$. Hence, $\CB_n(ik) \neq 0$ for $k \neq 0$ real.

\vspace{.05in}
$(ii)$ This is an immediate consequence of the corresponding symmetry of the function $J_n^2$.

\vspace{.05in}
$(iii)$ Using the series representation of the Bessel function it is straightforward to obtain the following asymptotic expansion, as $n \to \infty$:

\begin{equation} \label{J_n large n}
J_n(z) = \frac{(z/2)^n}{n!} \left[ 1 + O\left(n^{-1}\right) \right],
\end{equation}

\n uniformly for $z$ lying in a compact set of $\mathbb{C}$. Thus, $J_n \to 0$ uniformly on compact sets.

\vspace{.05in}
$(iv)$ Let $h_1, h_2 \in \RR$ be any numbers, consider the horizontal strip $H = \{k\in \mathbb{C}: h_1 < \Im k < h_2\}$. Let us show that

\begin{equation} \label{k B_n(k) limit}
\liminf_{\substack{\Re k \to + \infty \\ k \in H}} |k \CB_n(k)| >0.
\end{equation}

\n This will imply that any sequence of distinct zeros of $\CB_n$ cannot lie in any horizontal strip, concluding the proof. By part $(ii)$ the zeros of $\CB_n$ are symmetric about the imaginary axis, so we confined our attention to the right half-plane and in \eqref{k B_n(k) limit} assumed that $\Re k \to + \infty$.

Let us write $k=x + iy$, we will show that $|k \CB_n(k)|$ grows logarithmically, like $\ln x$, so that the quantity in \eqref{k B_n(k) limit} equals to infinity (when $d=3$, it just stays bounded from below by a positive constant). The large argument approximation of Bessel's function (\cite{handbook} 9.2.1) implies that for any $|\arg z|< \pi$

\begin{equation} \label{J_nu large z}
J_\nu(z) = \sqrt{\frac{2}{\pi z}} \left[ \cos\left(z-\tfrac{\pi \nu}{2} - \tfrac{\pi}{4}\right) + e^{|\Im z|} R(z) \right],    
\end{equation}

\n where $R(z)$ denotes the remainder term, whose dependence on $\nu$ is suppressed from the notation. Further, there exist constants $c, M > 0$, depending on $\nu$, such that

\begin{equation} \label{R bound}
|R(z)| \leq \frac{M}{|z|}, \qquad \forall |z|\geq c.
\end{equation}

\n Below, with a slight abuse of notation, we may change the remainder term from one line to another. For example, if in addition $z$ is restricted to a horizontal strip, the factor $e^{|\Im z|}$ can be absorbed into the remainder term. In fact, assume that $z$ lies in a horizontal strip inside the right half-plane:

\begin{equation} \label{H'}
H' = \{z \in \mathbb{C}: h_1'< \Im z < h_2', \ \Re z > 0\}
\end{equation}

\n for some constants $h_1',h_2'$. Squaring the representation \eqref{J_nu large z} and using that the cosine term stays bounded for $z \in H'$, we obtain

\begin{equation} \label{J_n^2 large z}
J_n^2(z) = \frac{2}{\pi z} \left[ \cos^2\left(z-\tfrac{\pi n}{2} - \tfrac{\pi}{4}\right) + R(z) \right]
\qquad \qquad z \in H',
\end{equation}

\n where $R(z)$ satisfies the estimate \eqref{R bound} with the constant $M$ depending on $n$ and $h_1', h_2'$. Assume that $x$ is large enough: $x \geq c$, where $c$ is as in \eqref{R bound}, and let us split the integral \eqref{B_n} defining $\CB_n$ into two parts: 

\begin{equation} \label{B_n two parts}
\CB_n(k) = \int_0^{\frac{c}{x}} f(t) J_n^2(k t) d t + \int_{\frac{c}{x}}^1 f(t) J_n^2(k t) d t.
\end{equation}

\n Henceforth, let $k \in H$. For $t \in (0,1)$ we have $z = kt \in H'$, where $H'$ is given by \eqref{H'} with constants $h_1', h_2'$ depending on $h_1,h_2$. Therefore, using \eqref{J_n^2 large z} in the second integral of \eqref{B_n two parts} along with the half-angle formula $2\cos^2 (\theta) = 1 + \cos(2\theta)$, we obtain

\begin{equation} \label{f_n three parts}
\begin{split}
\pi k \CB_n(k) &= \pi k  \int_0^{\frac{c}{x}} f(t) J_n^2(k t) d t + \int_{\frac{c}{x}}^1 \frac{f(t)}{t} d t + \int_{\frac{c}{x}}^1 \frac{f(t)}{t} \left[ \cos\left(2kt-\pi n - \tfrac{\pi}{2}\right) + R(kt) \right] dt =
\\
&=: I_1 + I_2 + I_3
\end{split}
\end{equation}

\n We next prove that $I_1, I_3$ are bounded and $I_2 \to +\infty$ as $x \to +\infty$, which will conclude the proof of \eqref{k B_n(k) limit}. Recall that $f(t)=t+1$ (cf. \eqref{f}). Clearly, $I_2$ grows logarithmically:

$$I_2 = \ln x + 1 - \frac{c}{x} - \ln c.$$

To bound $I_1$, we first note that for $t$ in the corresponding integration interval

\begin{equation*}
|kt| = t \sqrt{x^2 + y^2} \leq t (x + |y|) \leq c + |y| \leq c + \max\{|h_1|, |h_2|\}, 
\end{equation*}

\n Therefore, $|J_n(kt)|$ can be bounded by a constant depending only on $n, h_1, h_2$. Further, as $f(t)$ is a bounded function we get

$$|I_1| \leq C \frac{|k|}{x} \leq C \left( 1 + \frac{|y|}{x} \right) \leq C',$$

\n for some constants $C, C' >0$ depending on $n, h_1$ and $h_2$.

To bound $I_3$, we use that for $t$ in the corresponding integration interval $|kt| = t \sqrt{x^2 + y^2} \geq t x \geq c$ and hence the error bound \eqref{R bound} can be used with $z = kt$. Namely,

\begin{equation*}
\left| \int_{\frac{c}{x}}^1 \frac{f(t)}{t} R(kt) dt \right| \leq \int_{\frac{c}{x}}^1 \frac{f(t)}{t} \frac{M}{|k| t} dt \leq \frac{C}{|k|} \int_{\frac{c}{x}}^1 \frac{dt}{t^2} \leq \frac{C}{x} \left( \frac{x}{c} - 1  \right) \leq C',
\end{equation*}

\n for some constants $C, C'>0$. It remains to bound the term

\begin{equation*}
\int_{\frac{c}{x}}^1 \frac{f(t)}{t} \cos\left(2kt-\gamma\right) dt, 
\end{equation*}

\n where we set $\gamma = \pi n + \frac{\pi}{2}$ for shorthand. After integrating by parts this integral equals to

\begin{equation*}
\frac{1}{2k} \left[ f(1) \sin(2k - \gamma) - \frac{x}{c} f\left(\frac{c}{x}\right) \sin \left( 2 c \frac{k}{x} - \gamma \right) \right] - \frac{1}{2k} \int_{\frac{c}{x}}^1 \left( \frac{f(t)}{t} \right)' \sin (2kt -\gamma) dt. 
\end{equation*}

\n Using that sine is bounded when its argument lies in a horizontal strip in the complex plane, the modulus of the above quantity can be bounded by a constant multiple of 

\begin{equation*}
\frac{1}{|k|} (1+x) + \frac{1}{|k|} \int_{\frac{c}{x}}^1 \frac{dt}{t^2}. 
\end{equation*}

\n It follows now that the last quantity stays bounded.

\subsection{The expansion of $J_\nu(\nu z)$ and some properties of $\alpha$ and $\beta$} \label{SECT App alpha beta}

Let $\alpha, \beta$ be defined by \eqref{alpha beta}. For real $z$, introduce the functions

\begin{equation} \label{zeta x}
\zeta = \zeta(z) =
\begin{cases}
\left[ \tfrac{3}{2} \alpha(z) \right]^{\frac{2}{3}}, \qquad &0<z\leq 1
\\[.05in]
- \left[ \tfrac{3}{2} \beta(z) \right]^{\frac{2}{3}}, & 1 \leq z < \infty
\end{cases}
\end{equation}

\n and

\begin{equation} \label{b_0}
b_0(\zeta) = 
\begin{cases}
\displaystyle - \frac{5}{48 \zeta^2} + \frac{1}{\zeta^\frac{1}{2}} \left[ \frac{5}{24 (1-z^2)^\frac{3}{2}} - \frac{1}{8(1-z^2)^\frac{1}{2}} \right], \qquad \qquad &0<z<1
\\[.15in]
\displaystyle - \frac{5}{48 \zeta^2} + \frac{1}{(-\zeta)^\frac{1}{2}} \left[ \frac{5}{24 (z^2-1)^\frac{3}{2}} + \frac{1}{8(z^2-1)^\frac{1}{2}} \right], & z>1
\end{cases}
\end{equation}

\n Then, as $\nu \to +\infty$ (\cite{handbook} 9.3.35 - 9.3.42)

\begin{equation} \label{Bessel uniform}
J_\nu(\nu z) = \left( \frac{4 \zeta}{1-z^2} \right)^{\frac{1}{4}} \left[ \nu^{-\frac{1}{3}} \Ai(\nu^{\frac{2}{3}} \zeta) \left\{1 + O\left( \nu^{-2} \right) \right\} + \nu^{-\frac{5}{3}} b_0(\zeta) \Ai'(\nu^{\frac{2}{3}} \zeta) \left\{1 + O\left( \nu^{-2} \right) \right\} \right]
\end{equation}

\n uniformly for $z$ inside the sector $|\arg z| \leq \pi - \delta$, where $\delta>0$ is any small number, $O$'s are uniform in $z$, $\zeta = \zeta(z)$ is the analytic continuation of the function \eqref{zeta x} to the complex plane cut along the negative real axis and $b_0(\zeta)$ is also defined by analytic continuation. We mention that $\zeta$ is analytic near $1$ and has the expansion:

\begin{equation} \label{zeta asymp 1}
\zeta(z) = \sqrt[3]{2} (1-z) + \frac{3 \sqrt[3]{2}}{10} (1-z)^2 + \frac{32\sqrt[3]{2}}{175} (1-z)^3 + ..., \qquad \qquad z \to 1
\end{equation}

\n Further, the coefficient $b_0(\zeta)$ is analytic near $\zeta=0$, i.e. $z=1$. This is not obvious from the above formula of $b_0$. A delicate cancellation happens in the above representation, where the three terms in the expansion \eqref{zeta asymp 1} are used to get 

\begin{equation} \label{b_0 at 0}
b_0(0) = \frac{\sqrt[3]{2}}{70}.
\end{equation}

\n In our applications $z$ lies inside the right half-plane and we consider two cases depending on whether $\Re z$ is larger, or smaller than 1. It is straightforward to see that

\begin{equation} \label{zeta z<1}
\zeta(z) = \left[ \tfrac{3}{2} \alpha(z) \right]^{\frac{2}{3}},
\qquad \qquad z \in V_0 = \{0<\Re z <1\}
\end{equation}

\n and that $b_0(\zeta)$ is given by the first formula of \eqref{b_0}, where $z \in V_0$ is complex now. Similarly,

\begin{equation} \label{zeta z>1}
\zeta(z) = - \left[ \tfrac{3}{2} \beta(z) \right]^{\frac{2}{3}},
\qquad \qquad z \in V_1 = \{\Re z > 1\}
\end{equation}

\n and $b_0(\zeta)$ is given by the second formula of \eqref{b_0}. Below we collect some properties of the functions $\alpha, \beta$ that are used in Sections~\ref{SECT n<k} and \ref{SECT n comparable k}.

\begin{lemma} \label{LEM alpha beta prop} \mbox{}
\begin{enumerate}

\item[(i)] There exists a constant $C>0$ such that

\begin{equation*}
|\Im \beta(z)| \leq C |\Im z|, \qquad \qquad z \in V_1.
\end{equation*}

\item[(ii)] Let $\epsilon \in (0,1)$, then there exist constants $c=c_\epsilon>0$ and $\delta=\delta_\epsilon>0$ such that

\begin{equation} \label{alpha lower real}
\Re \alpha(z) \geq c \left(1+\ln \frac{1}{|z|} \right) , \qquad \qquad z \in V_0, \quad |z|<1-\epsilon \quad \text{and} \quad |\Im z| < \delta.
\end{equation}

\end{enumerate}

\end{lemma}

\begin{proof}
$(i)$ Let $z\in V_1$, simple calculation shows that $\beta'(z) = \sqrt{z^2-1} / z$. Fix a small number $\delta>0$ and write

\begin{equation*}
\beta(z) = \beta(1+\delta) + \int_{1+\delta}^z \frac{\sqrt{\xi^2-1}}{\xi} d\xi.
\end{equation*}

\n Choose the contour of integration to be the horizontal line segment $[1+\delta, x]$ followed by the vertical line segment $x+i\left[0,|y|\right]$, where $z=x+iy$. Letting $\delta \to 0$ in the above formula we arrive at

\begin{equation} \label{beta intrep}
\beta(x+iy) = \int_{1}^x \frac{\sqrt{t^2-1}}{t} dt + i \int_0^{y} \frac{\sqrt{(x+it)^2-1}}{x+it} dt .
\end{equation}

\n Taking imaginary parts, the first integral term drops and in the second one we use that the integrand $\sqrt{\xi^2-1}/\xi$ is a bounded function in $V_1$. This concludes the proof.

$(ii)$ Let $z\in V_0$ and $|z| < 1-\epsilon$. First, note that $\Re \alpha(z) \sim - \ln |z|$ as $z \to 0$, therefore, there exists a constant $c_1>0$ such that

\begin{equation} \label{re alpha small z}
\Re \alpha(z) \geq \frac{1}{2} \left( 1 + \ln \frac{1}{|z|} \right), \qquad \qquad |z|\leq c_1.  
\end{equation}

\n Analogously to the integral representation \eqref{beta intrep} for $\beta$, using $\alpha'(z) = -\sqrt{1-z^2} / z$ we obtain the following for $\alpha$:

\begin{equation*} 
\alpha(x+iy) = \int_x^1 \frac{\sqrt{1-t^2}}{t} dt - i \int_0^{y} \frac{\sqrt{1-(x+it)^2}}{x+it} dt. 
\end{equation*}

\n The first integral above is real and as $x = \Re z <1-\epsilon$ we get

\begin{equation} \label{re alpha int lower bound}
\Re \alpha(x+iy) \geq \int_{1-\epsilon}^1 \frac{\sqrt{1-t^2}}{t} dt + \Im \int_0^{y} \frac{\sqrt{1-(x+it)^2}}{x+it} dt. 
\end{equation}

\n Now let us assume $|z|>c_1$ and $|y| \leq c_1/2$, so that $x \geq c_1/2$. Then

\begin{equation*} 
\left| \Im \int_0^{y} \frac{\sqrt{1-(x+it)^2}}{x+it} dt \right| \leq \int_0^{|y|} \frac{\sqrt{2}}{\sqrt{x^2+t^2}} dt \leq \frac{2\sqrt{2}}{c_1} |y|.
\end{equation*}

\n We can therefore choose $|y|$ so small that the above bound is less than half of the first integral term on the right hand side of \eqref{re alpha int lower bound}. Therefore, there exist positive constants $c=c(\epsilon)$ and $\delta = \delta(c_1, \epsilon)$, such that

\begin{equation} \label{re alpha z>c_1}
\Re \alpha(z) \geq c, \qquad \qquad z\in V_0, \quad c_1<|z|<1-\epsilon, \quad |\Im z|<\delta.
\end{equation}

\n It is now evident that combining the bounds \eqref{re alpha small z} and \eqref{re alpha z>c_1} we can conclude the proof.

\end{proof}

{\footnotesize
\bibliographystyle{plain}
\bibliography{refs}}

\end{document}